\documentclass[a4paper]{amsart}

\usepackage[pdftex]{graphicx}
\usepackage[pdftex]{xcolor}

\usepackage{graphicx}

\usepackage{mathtools,  stmaryrd}
\usepackage{xparse} \DeclarePairedDelimiterX{\Iintv}[1]{\llbracket}{\rrbracket}{\iintvargs{#1}}
\NewDocumentCommand{\iintvargs}{>{\SplitArgument{1}{,}}m}
{\iintvargsaux#1} %
\NewDocumentCommand{\iintvargsaux}{mm} {#1\mkern1.5mu,\mkern1.5mu#2}
\usepackage{hyperref}

\makeatletter
\newtheorem*{rep@theorem}{\rep@title}
\newcommand{\newreptheorem}[2]{%
\newenvironment{rep#1}[1]{%
 \def\rep@title{#2 \ref{##1}}%
 \begin{rep@theorem}}%
 {\end{rep@theorem}}}
\makeatother

\usepackage{verbatim}
\usepackage{amsmath}
\usepackage{amssymb, mathrsfs}
\usepackage{amsbsy}

\usepackage{amscd}
\usepackage{amsthm}

\usepackage[english]{babel}
\usepackage{todonotes}
\usepackage[T1]{fontenc}

\usepackage[normalem]{ulem}

\newcommand{\N}{\mathbb{N}}

\newcommand{\E}{\mathbb{E}}

\renewcommand{\P}{\mathbb{P}}

\newcommand{\kC}{\mathcal{C}}

\newcommand{\kF}{\mathcal{F}}

\newcommand{\kE}{\mathcal{E}}

\newcommand{\rmP}{\mathrm{P}}

\newcommand{\rme}{\mathrm{E}}

\newcommand{\lin}{\left[\kern-0.15em\left[}
\newcommand{\rin} {\right]\kern-0.15em\right]}
\newcommand{\linf}{[\kern-0.15em [}
\newcommand{\rinf} {]\kern-0.15em ]}
\newcommand{\ilin}{\left]\kern-0.15em\left]}
\newcommand{\irin} {\right[\kern-0.15em\right[}

\newcommand{\cvlaw}{\stackrel{{ (d)}}{\longrightarrow}}

\def\al#1{\begin{align*}#1\end{align*}}
\def\aln#1{\begin{align}#1\end{align}}

\usepackage{constants}

\newconstantfamily{c}{symbol=c}

\newconstantfamily{a}{symbol=\alpha}

\newcommand{\secno}[1]{\thesection.\arabic{#1}}
\newconstantfamily{kE}{
symbol=\mathcal{E},
format=\secno,
reset={section}
}

\renewcommand{\tilde}{\widetilde}

\newtheorem{lem}{Lemma}[section]
\newtheorem{remark}[lem]{Remark}
\newtheorem{prop}[lem]{Proposition}
\newtheorem{thm}[lem]{Theorem}

\newtheorem{cor}[lem]{Corollary}

\newcounter{assu}
\usepackage{color}
\definecolor{lilas}{RGB}{182, 102, 210}

\newcommand{\CC}{\color{blue}}

\renewcommand{\CC}[1]{\textcolor{blue}{#1}}

\numberwithin{equation}{section}

\title[maximum of 2D Gaussian directed polymer in the subcritical regime]
{The maximum of the two dimensional Gaussian directed polymer in the subcritical regime}
\date{\today}

\author{Cl\'ement Cosco} 
\address[Cl\'ement Cosco]
{Université Paris Dauphine, Paris, France}
\email{cosco@ceremade.dauphine.fr}

\author{Shuta Nakajima} 
\address[Shuta Nakajima]
{University of Meiji, Kanagawa, Japan}
\email{njima@meiji.ac.jp}

\author{Ofer Zeitouni} 
\address[Ofer Zeitouni]
{Weizmann Institute, Rehovot, Israel}
\email{ofer.zeitouni@weizmann.ac.il}

\begin{document}

\begin{abstract} We study the maximum $\phi_N^*$ of the partition function
of the
two dimensional (subcritical) Gaussian directed polymer over an  $\sqrt N \times \sqrt N$ box.
We show that $\phi_N^*/\log N$ converges towards a constant $\sigma^*$, which we identify to be the same as for the maximum of a branching random walk with a slowly varying variance profile as studied in \textit{Fang-Zeitouni, J. Stat. Phys. 2012} and (in the context of the generalized random energy model) in 
\textit{Bovier-Kurkova, Ann. Inst. H. Poincare 2004}.
    \end{abstract}
    \maketitle

\section{Introduction}
We consider in this paper the partition function of the
two dimensional (subcritical) Gaussian directed polymer, and analyze its maximum over a suitable box. We begin by introducing the model.
\subsection{Model and background}
Let $(\omega(i,x))_{i\in \N,x\in\mathbb{Z}^2}$ be i.i.d. standard Gaussian random variables, that is  such that 
\[
\mathbb{E}[\omega(i,x)]=0,\,\mathbb{E}[\omega(i,x)^2]=1.
\]
We use $\P$ and $\E$ to denote the probability law (respectively, the expectation) of the collection of random variables $\{\omega(i,x)\}_{i\in \mathbb{N},x\in \mathbb{Z}^2}$.

Let $\{S_n\}_{n\in \mathbb{Z}_+}$ denote a simple random walk on $\mathbb{Z}^2$. We write 
${\rm E}_x$ for expectations with respect to the random walk
$\{S_n\}$ with $S_0=x$. Given $\hat{\beta}>0$, we set 
$$\beta_N:=\frac{\hat{\beta}}{\sqrt{ R_N}},\,R_N:= {\rm E}^{\otimes 2}_0\Big[\sum_{k=1}^N \mathbf{1}_{S_k^1=S_k^2}\Big] \sim  \frac{\log N}{\pi},$$
where here and throughout for a positive integer $q$ and with ${\bf x}=(x_1,\ldots,x_q)\in (\mathbb{Z}^2)^q$, we write ${\rm E}^{\otimes q}_{\bf x}$  for the expectation with respect to $q$ independent simple random walks 
$\{S_n^i\}_{n\in \mathbb{Z}_+, i\in \Iintv{1,q}}$ satisfying $S_0^i=x_i$ (and we used
$\Iintv{1,q} := \{1,2\dots,q\}$).
We define the partition function of the directed polymer:
\begin{equation}
\label{eq-180924a}
Z_N(x):={\rm E}_{x} \left[e^{\sum_{i=1}^N \{\beta_N\omega(i,S_i)-\beta_N^2/2\}}\right],\qquad W_N(x):=\log Z_N(x),
\end{equation}
and set
\begin{equation}
    \label{eq-180924b}
\phi_N(x):= \sqrt{\log{N}} \log Z_N\left([ x\sqrt N ] \right).
\end{equation}
(Throughout the paper, we consider $\hat\beta$ fixed and omit it from the notation.)
For background, motivation and results on the rich theory
surrounding this topic (in various dimensions), we refer the reader to \cite{CStFlour,Z24}.
In particular, we mention the relation with the $2$ dimensional stochastic
heat equation (SHE).

The particular rescaling $\beta_N$ 
was discovered in the context of the SHE by Bertini and Cancrini \cite{Bertini98} and was later generalized by Caravenna, Sun and Zygouras \cite{CaraSuZy-universalityrelev}, in both the SHE and polymer setups, to a wider range of parameters for which a phase transition occurs. See also \cite{CaSuZy18,CaSuZyCrit21,ChDu18,Gu18KPZ2D,NaNa21}. In particular, it follows from these works that for $\hat\beta<1$, the variable $\log Z_N(x)$ converges in distribution to a normal variable of mean
$- {\lambda^2}/ 2$ and variance $\lambda^2$, where
$\lambda^2=\lambda^2(\hat\beta)=-\log (1-{\hat \beta}^2)$.

Our focus in this article is 
in the spatial behavior of $W_N(x)$.
Indeed, one has, see \cite{CaSuZy18}, that
\begin{equation} \label{eq:GFFlimit}
G_N(x):=\phi_N(x) -\E\phi_N(0)
\cvlaw \sqrt{\frac {\hat {\beta}^2} {1-{\hat \beta}^2}} G(x),
\end{equation}
with $G(x)$ a log-correlated Gaussian field on $\mathbb R^2$. 
The convergence in law
of \eqref{eq:GFFlimit} is in the weak sense, i.e.\@
for any smooth, compactly supported function  $\psi$, the random variable
$\int \psi(x) G_N(x) dx$ converges in distribution to a centered
Gaussian random
variable of variance $\hat{\beta}^2 \sigma^2_\psi/(1-\hat{\beta}^2)$, where
\begin{equation}
  \label{eq:lim-cov}
  \sigma^2_\psi :=\tfrac1{\pi} \iint \psi(x)\psi(y)\int_{\|x-y\|^2/2}^\infty
  z^{-1}e^{-z} dz.\end{equation}
  One recognizes $\sigma^2_\psi$ in \eqref{eq:lim-cov}
  as the variance of the integral of
  $\phi$ against the solution of the \textit{Edwards-Wilkinson} equation. For a related result in the KPZ/SHE setup, see \cite{CaSuZy18,Gu18KPZ2D,NaNa21}.

  Logarithmically correlated fields, and
  in particular their extremes and large values,
  have played an important recent
  role in the study of various models of probability theory at the critical
  dimension,
  ranging from their own study \cite{Biskup,BDZ,DRSV,RV},
  random walk and Brownian motion
  \cite{BRZ,DPRZ}, random matrices \cite{CMN,CN,CFLW},
  Liouville quantum
  gravity \cite{DuSh,KRV}, turbulence \cite{GRV}, and more.
  In particular, exponentiating Gaussian
  logarithmically
  correlated fields yields Gaussian multiplicative chaos,
  with the ensuing question of convergence towards them.

 In the context of polymers,
\eqref{eq:GFFlimit} opens the door to the study of such questions.
A natural role is played by the random measure
\[\mu_{N}^{\gamma}(x):=\frac{e^{\gamma G_N(x)}}
{\E e^{\gamma G_N(x)}},\]
and it is natural to ask about its convergence towards a Gaussian Multiplicative Chaos, and
about extremes of $\phi_N(x)$ for $x$ in some compact subset of $\mathbb R^2$, and specifically on $[-1,1]^2$.
The purpose of this paper is to answer the latter question at the level of
leading order convergence. We describe next our results.

\subsection{Main results}
Set {$\phi_N^*:=\sup_{x\in[-1,1]^2}\phi_N(x)$.}
\begin{thm}
\label{theo-main}
For any $\hat{\beta}<1$, we have, in probability, 
\begin{equation}
    \label{eq-180924d}
\lim_{N\to\infty} \frac{\phi_N^*}{\log N}= \sqrt 2\int_{0}^1\sigma(u){\rm d}u=:\bar \sigma^*,
\end{equation}
where
\begin{equation}
    \label{eq-180924c}
\sigma(u) := \sqrt{\frac{\hat\beta^2}{1-\hat\beta^2 u}}.
\end{equation}
\end{thm}

It is worth commenting on the limit in \eqref{eq-180924d}. A naive guess
based on the exponential moments of $\phi_N(x)$ computed in \cite{cosco2021momentsUB} would lead, by standard first moment arguments, to upper bounding the left side of
\eqref{eq-180924d} by the expression
\begin{equation}
    \label{eq-180924e}
\left(\int_{0}^1\sigma(u)^2{\rm d}u\right)^{1/2}.\end{equation}
In fact, the lower bound on the exponential moments of $\phi_N(x)$ contained in \cite{cosco2023momentsLB} shows that the upper bound on exponential moments
is tight. The reason for the discrepancy between \eqref{eq-180924d}
and \eqref{eq-180924e} lies in a certain inhomogeneity in the contributions to
$Z_N(x)$ coming from different scales. We return to this point below, but mention at this stage that this situation is very similar to
what happens for branching  Brownian motion
with time dependent variance, as analyzed in \cite{FangZeitouni12}
and \cite{MaillardZeitouni16}, where the expression in the right hand side of \eqref{eq-180924d} appears;
see also \cite{BovierKurkova04} for a related result. Thus, the rescaled two dimensional polymer model seems to be the first natural model in which such a varying variance profile appears naturally.

\subsection{Notation}
Throughout, we use the following notation.

We define $\Iintv{a,b}$ as $[a,b]\cap \mathbb{Z}$ for any $a<b$. 
The notation $\log_+ x$ is used to denote $\max\{\log x,0\}$.

For integers $s \le t$, the polymer partition function between times $s$ and $t$, starting at position $x$, is given by
\[
    Z_{s,t}(x) := \mathrm{E}_x \Biggl[\exp\Biggl(\sum_{i\in \Iintv{s,t}} \bigl(\beta_N\omega(i,S_i) - \tfrac{\beta_N^2}{2}\bigr)\Biggr)\Biggr].
\]
In particular, setting $s=1$ yields
\[
    Z_{t}(x) := Z_{1,t}(x).
\]

Given $t\in \mathbb{N}$, the endpoint distribution $\mu_t(x,y)$ at time $t$ is defined as
\[
    \mu_t(x,y) 
    := Z_t(x)^{-1} \mathbb{E}_{x}\Biggl[\exp\Biggl(\sum_{i=1}^t \bigl(\beta_N\omega(i,S_i) - \tfrac{\beta_N^2}{2}\bigr)\Biggr)\mathbf{1}_{\{S_t=y\}}\Biggr].
\]
Here, $\mu_t(x,y)$ represents the probability that the polymer, starting at $x$, ends at $y$ at time $t$ under the associated random path measure.

For $M,N\in\mathbb{N}$ and $k\in \Iintv{1,M}$, we define
\[
    t_k := t_{k,M,N} := \lceil N^{k/M}\rceil,
    \qquad
    r_k := r_{k,M,N} := \lceil N^{k/(2M)}\rceil,
\]
and set $r_0:=r_{-1}:=1$. The box centered at $x$ with radius $r$ is denoted by $\Lambda_r(x)$, defined as
\[
    \Lambda_r(x) := [x-r,x+r]^2 \cap \mathbb{Z}^2.
\]
We write $\Lambda_{n}:=\Lambda_{n}(0)$.

Lastly, we define 
\begin{equation} \label{eq:def_lambda_T_Nbis}
    \lambda_{u,v}^2 := \log\left(\frac{1-\hat{\beta}^2 u}{1-\hat{\beta}^2 v}\right).
\end{equation}

\subsection{Outline of the proof}

Let $M$ be a sufficiently large integer. For $k\in \Iintv{1,M}$, we define
\[
W_{k,N}(x):= \frac{Z_{t_k}(x)}{Z_{t_{k-1}}(x)},
\]
so that $\log{Z_N(x)}=\sum_{k=1}^M \log{W_{k,N}}(x)$. 

The proof of the upper bound in Theorem \ref{theo-main} proceeds as follows.
A natural strategy (see \cite{Kistler} for a nice exposition) is to replace the event $\log{Z_N(x)}\geq (1+\epsilon) \bar\sigma^* \sqrt{\log N}$ by 
the intersections (over $\ell\in \Iintv{1,M}$) of the events $\log_+ W_{\ell,N}(x)\geq  \alpha_\ell$, for appropriately chosen $\alpha_\ell$.
Assuming for the moment that this can be done, and further that
the terms $\log{W_{k,N}}(x)$ were independent (which they are not) and
that they 
are exponentially equivalent to their maximum over boxes of side length 
$r_\ell=\lfloor N^{\ell/2M}\rfloor$,
the upper bound would then follow from a first moment computation, assuming
the knowledge of exponential moments of $\sqrt{\log N}\log{W_{k,N}}(x)$. The latter are evaluated in \cite{CN24} (improving the estimates in \cite{cosco2021momentsUB}) and recalled in Appendix
\ref{app-C}. The choice of the parameters $\alpha_\ell$ follows closely 
the approach in \cite{FangZeitouni12} and follows from the analysis
of their variational problem that is presented in Lemma 
\ref{lem: variationalProblems}.

To carry on this strategy, three steps are needed. The first, contained
in Proposition \ref{prop: first goal for upper bound}, is essentially
a book-keeping exercise together with a union bound, and formalizes a notion of barrier, see \eqref{eq-Aell} and \eqref{eq-Bar}. To handle both the decoupling and the continuity issue mentioned above, we represent expectations of the partition function over a certain time interval in terms
of expectation with respect to the polymer Gibbs measures $\mu_t(x,y)$.
We identify an event $\mathcal{E}_{\ell,N}$ under which this Gibbs measure
is well dominated by a Gaussian, see
\eqref{Eq: max W}, and show (by means of a local CLT) that this event has
large enough probability, see Lemma
\ref{lem: decay of kE}; In the proof, a crucial role is played by an estimate of
moments for point-to-point partition functions presented in
Appendix \ref{app-A}. The event $\mathcal{E}_{\ell,N}$ ensures the continuity in terms of the 
starting point, and together with a coarse graining step originating in 
\cite{CN24}, allows for the decoupling alluded to before,
see Lemmas \ref{lem: decay of kE}, \ref{lem: L business bound} and \ref{lem:combLem}. 

Turning to the lower bound, this is based on a second moment computation, using the same barrier decomposition (this time,  counting 
only points $x$ so that $\log_+ W_{\ell,N}(x)\geq  \alpha_\ell$ for all $\ell\leq (1-\varepsilon_1)M$, that is, we disregard contributions to the partition function from the time interval $\Iintv{N^{1-\varepsilon_1},N}$ with some small parameter $\varepsilon_1>0$). 
At a first step, we introduce approximations to $W_{k,N}(x)$, called
$\widetilde{W}_{k,N}(x)$ (see \eqref{eq-tildeW}), which basically performs averaging with respect to 
(a restriction of) the polymer Gibbs measure $\mu_{t_{k-1}}$ of the contribution to the partition function during the interval $\Iintv{t_{k-1},t_k}$. 
We then show in Lemma \ref{lem: convert W to tilde-W} that this replacement is not creating an error, even if a supremum over all starting points is 
taken. The advantage of working with $\widetilde{W}_{k,N}(x)$ is that one has a localization of the contribution of the disorder 
between times $t_{k-1}$ and $t_k$. This localization  then creates, much as in the upper bound, enough decoupling so that the second moment method can apply, see Lemma \ref{lem: key lemma for lower bound}. (The necessary lower bound on the first moment is obtained by a lower bound on joint exponential moments, see Lemma \ref{lem: lower bound of moments}.) Finally, we show that adding the contribution to the partition function during the time interval $\Iintv{t_{k-1},t_k}$ cannot decrease it much, see \eqref{eq-151024}. This last step uses an a-priori estimate on the lower tail of the partition function contained in Appendix \ref{app-D}.

\section{Preliminaries}
\subsection{Variational problems}
\begin{lem}\label{lem: variationalProblems}
 Given $M\in\N$, $a\geq 0$, a non-decreasing function $f:\Iintv{1,M}\to (0,\infty)$, and $t\in \Iintv{1,M}$, consider the following set 
  $$\mathcal{A}_{a,M}(t):=\left\{g:\Iintv{t,M}\to [0,\infty):
  \!\!\! \begin{array}{c}
\sum_{u\in \Iintv{t,M}}g(u)> a+\sum_{u\in \Iintv{t,M}}f(u),\\
\forall s>t,\sum_{u\in \Iintv{s,M}}g(u)\leq a+ \sum_{u\in \Iintv{s,M}}f(u)
 \end{array}
 \!\right\}.$$
  Then,
  \[\inf_{g\in\mathcal A_{a,M}(t)} \sum_{s\in \Iintv{t,M}} \frac{g(s)}{f(s)} \geq M-t+1 + \frac{a}{f(M)}.\]
\end{lem}
\begin{proof}
 Assume $g\in \mathcal{A}_{a,M}(t)$. We write 
$$F(u):=\sum_{s\in\Iintv{u,M}}  f(s) \;\; \text{and}\;\; G(u):=\sum_{s\in\Iintv{u,M}}  g(s).$$ 
 We can estimate using summation by parts as follows:
\aln{
\sum_{u\in \Iintv{t,M}} \frac{g(u)}{f(u)}&=\sum_{u\in \Iintv{t,M}}  \frac{(G(u)-G(u+1))}{f(u)}\notag\\
&=\frac{G(t)}{f(t)}+ \sum_{u\in\Iintv{t+1,M}}  G(u) (f(u)^{-1}-f(u-1)^{-1}).\label{Eq: IntegrationbyParts}
} If $f$ is non-decreasing and $g\in \mathcal{A}_{a,M}(t)$, then $F(t)+a\leq G(t)$ and $F(s)+a\geq G(s)$ for $s>t$, which implies that the result is non-negative. Therefore, we can further bound this from below by 
\al{
&\frac{F(t)+a}{f(t)}+ \sum_{u\in\Iintv{t+1,M}}  (F(u)+a) (f(u)^{-1}-f(u-1)^{-1})\\
&= \sum_{u\in \Iintv{t,M}} \frac{f(u)}{f(u)} + \frac{a}{f(M)}= M-t+1+\frac{a}{f(M)},
}
where we have used \eqref{Eq: IntegrationbyParts} with $f$ in place of $g$. Therefore, we conclude that
\[\inf_{g\in\mathcal A_M(t)} \sum_{s\in \Iintv{t,M}} \frac{g(s)}{f(s)}\geq M-t+1+\frac{a}{f(M)}.\]
\end{proof}
\subsection{Technical lemmas}
We note  that 
$$W_{k,N}(x)=\sum_{y\in \mathbb{Z}^2} \mu_{t_{k-1}}(x,y)\theta_{t_{k-1}} Z_{t_{k}-t_{k-1}}(y),$$
where $\theta_t$ is the shift in time $t$, that is,  $\theta_t \omega(s,x):=\omega(t+s,x)$ and $\theta_t X(\omega) := X(\theta_t \omega)$ for a random variable $X$, and $\mu_t(x,y)$ denotes the Polymer Gibbs measure at time $t$, that is, 
$$\mu_t(x,y) : = \frac{{\rm E}_x \left[e^{\sum_{i\in \Iintv{s,t}} \{\beta_N\omega(i,S_i)-\beta_N^2/2\}}\mathbf{1}_{\{S_t=y\}}\right]}{Z_t(x)}.$$
Let us further define
\[
    {\widetilde{Z}_{s,t}(x):={\rm E}_x \left[e^{\sum_{i\in \Iintv{s,t}} \{\beta_N\omega(i,S_i)-\beta_N^2/2\}}\mathbf{1}_{\{\forall i\in \Iintv{s,t},\,|S_i|\leq  \sqrt{t}\log N\}}\right],
    }
\]
and 
 {$\widetilde{Z}_{t}(x):=\widetilde{Z}_{1,t}(x)$}. 
 For $k\in \Iintv{1,M}$, we define
\begin{equation}
    \label{eq-tildeW}
\widetilde{W}_{k,N}(x):=\sum_{\|y-x\|\leq r_{k-1}\log N} \mu_{t_{k-1}}(x,y)\,  \theta_{t_{k-1}} \widetilde{Z}_{t_{k}-t_{k-1}}(y).
\end{equation}
\begin{lem}\label{lem: convert W to tilde-W}
    There exists $C=C(M,\hat{\beta})>0$ such that for any $k\in \Iintv{1,M}$ and $x\in \mathbb{Z}^2$,
    \aln{
    &\mathbb{E}[W_{k,N}(x)-\widetilde{W}_{k,N}(x)]\leq Ce^{-(\log N)^2/C},\\
    &\mathbb{P}(W_{k,N}(x)/\widetilde{W}_{k,N}(x)\geq 2)\leq C e^{-(\log N)^2/C}.
    }
    In particular, since $\log{Z_N(x)} = \sum_{k=1}^M \log W_{k,N}(x)$, we have
\aln{\label{lem:convert conclusion}
\mathbb{P}\left(\!\max_{x\in \Lambda_{\sqrt N}} \left|\log{Z_N(x)} - \sum_{k=1}^M \log \widetilde{W}_{k,N}(x) \right| \geq  M \log{2}\!\right)\!\leq C M (2\sqrt{N}+1)^2 e^{-(\log N)^2/C}. 
}
\end{lem}
\begin{proof}
Without loss of generality, we set $x=0$ as the distributions of $W$ and $\widetilde{W}$ are invariant under shifts. We write $Z_N,W_{k,N},\widetilde{W}_{k,N}$ for $Z_N(0),W_{k,N}(0),\widetilde{W}_{k,N}(0)$.    Note that $\mu_k(x,y) = p_k(x,y)  Z_k(x,y)/Z_k(x) $ with
    \begin{equation*}
        Z_k(x,y) := \rme_x\left[e^{ \sum_{i=1}^k \{\beta_N\omega(i,S_i)-\beta_N^2/2\}}\middle|S_k=y\right].
    \end{equation*}
    Hence, since $(\mu_{t_{k-1}}(0,y))_{y\in \mathbb{Z}^2}$ and $(\theta_{t_{k-1}} Z_{t_{k}-t_{k-1}}(y),\theta_{t_{k-1}} \tilde{Z}_{t_{k}-t_{k-1}}(y))_{y\in \mathbb{Z}^2}$ are independent, we estimate
    \al{
    &\mathbb{E}[W_{k,N}-\widetilde{W}_{k,N}]\\
    &\leq \sum_{|y|> r_{k-1}\log N} \mathbb{E}[\mu_{t_{k-1}}(0,y)] + \sum_{y\in \mathbb{Z}^2} \mathbb{E}[\mu_{t_{k-1}}(0,y)] \mathbb{E}[Z_{t_{k}-t_{k-1}}(y)-\widetilde{Z}_{t_{k}-t_{k-1}}(y))]\\
    &\leq  \sum_{|y|> r_{k-1}\log N} \mathbb{E}[\mu_{t_{k-1}}(0,y)] + {\rm P}_0\left(  \exists i\in \Iintv{1,t_k-t_{k-1}},\,S_i \notin \Lambda_{\sqrt{t_k-t_{k-1}}\log N}\right).
    }
 By the local central limit theorem of the simple random walk \cite[Theorem~2.3.11]{LL10}, the second term on the right-hand side is bounded by $C e^{-(\log N)^2/C}$ with some $C>0$; by Corollary~\eqref{eq: lower tail2}, the first term on the right-hand side is bounded by
    \al{
     &\sum_{\|y\|\geq r_{k-1}\log N} \mathbb{E}[\mu_{t_{k-1}}(0,y)\mathbf{1}_{\{Z_{t_{k-1}}> N^{-1}\}}]  +      \sum_{\|y\|\geq r_{k-1}\log N} \mathbb{E}[\mu_{t_{k-1}}(0,y)\mathbf{1}_{\{Z_{t_{k-1}}\leq  N^{-1}\}}] \\
       &{\leq N \sum_{\|y\|\geq r_{k-1}\log N} \mathbb{E}[p_{t_{k-1}}(0,y)  Z_{t_{k-1}}(0,y)]  +      \sum_{\|y\|\geq r_{k-1}\log N} \mathbb{E}[\mathbf{1}_{\{Z_{t_{k-1}}\leq  N^{-1}\}}]}\\
     &=  N \sum_{\|y\|\geq r_{k-1}\log N} p_{t_{k-1}}(0,y)  +      \sum_{\|y\|\geq r_{k-1}\log N} \mathbb{P}(Z_{t_{k-1}}\leq N^{-1})\leq C e^{-(\log N)^2/C},
    }
which yields the first claim.   Moreover, we see that if $Z_{t_{k-1}}\geq 2 N^{-1}$, then
    \al{
      W_{k,N}-  \widetilde{W}_{k,N}&\leq N \sum_{\|y\|\geq r_{k-1}\log N} p_{t_{k-1}}(0,y) Z_{t_{k-1}}(0,y)\,\theta_{t_{k-1}} Z_{t_{k}-t_{k-1}}(y)\\
      & \;+ \!N\sum_{y\in \mathbb{Z}^2} p_{t_{k-1}}(0,y) Z_{t_{k-1}}(0,y) (\theta_{t_{k-1}} Z_{t_{k}-t_{k-1}}(y)-\theta_{t_{k-1}}\widetilde{Z}_{t_{k}-t_{k-1}}(y)) .
    }
Hence, we have 
   \al{  
    &\mathbb{P}(W_{k,N}\geq 2 N^{-1},\,\widetilde{W}_{k,N}\leq 1/N,\,Z_{t_{k-1}}\geq 2 N^{-1})\\
    &\leq \mathbb{P}(W_{k,N}-\widetilde{W}_{k,N}\geq 1/N,\,Z_{t_{k-1}}\geq 2 N^{-1})\\
    &\leq N \mathbb{E}[(W_{k,N}-\widetilde{W}_{k,N})\mathbf{1}_{Z_{t_{k-1}}\geq 2 N^{-1}}]\\
    &\leq N^2 \mathbb{E}\left[\sum_{\|y\|\geq r_{k-1}\log N} p_{t_{k-1}}(y) Z_{t_{k-1}}(y)\,\theta_{t_{k-1}} Z_{t_{k}-t_{k-1}}(y)\right]\\
    &\qquad +N^2\E\left[ \sum_{y\in \mathbb{Z}^2} p_{t_{k-1}}(y) Z_{t_{k-1}}(y) (\theta_{t_{k-1}} Z_{t_{k}-t_{k-1}}(y)- \theta_{t_{k-1}} \widetilde{Z}_{t_{k}-t_{k-1}}(y))\right].
    }
By the local central limit theorem for the simple random walk \cite[Theorem~2.3.11]{LL10}, this can be further bounded from above by 
   \al{
    & N^2 \sum_{\|y\|\geq r_{k-1}\log N} p_{t_{k-1}}(y) \\
    &\qquad + N^2  {\rm P}_0\left(  \exists i\in [t_k-t_{k-1}],\,S_i \notin [-\sqrt{t_k-t_{k-1}}\log N,\sqrt{t_k-t_{k-1}}\log N]^2\right)\\
    &\quad \leq C e^{-(\log N)^2/C},
    }
  with some $C>0$, where  we have used  the independence in time of the environment and $\mathbb{E}[Z_{n}]=1$. By Corollary~\eqref{eq: lower tail2}, we have
   \al{
   &  \mathbb{P}(\widetilde{W}_{k,N}\leq 1/N)\\ 
   & \leq      \mathbb{P}(\widetilde{W}_{k,N}\leq 1/N,\,W_{k,N}\geq 2 N^{-1},\,Z_{t_{k-1}}\geq 2 N^{-1}) \\
   &\qquad\qquad\qquad + \mathbb{P}(W_{k,N}< 2 N^{-1})
  +\mathbb{P}(Z_{t_{k-1}}< 2 N^{-1})\\
     & \leq 3 C e^{-(\log N)^2/C}. 
   }
   Therefore, since $\widetilde{W}_{k,N}> 1/N$ and $W_{k,N}-\widetilde{W}_{k,N} <  1/N$ imply $W_{k,N}/\widetilde{W}_{k,N}< 2$, by the first claim, we have
\al{
\mathbb{P}(W_{k,N}/\widetilde{W}_{k,N}\geq 2) & \leq      \mathbb{P}(\widetilde{W}_{k,N}\leq 1/N)+ \mathbb{P}(W_{k,N}-\widetilde{W}_{k,N}\geq 1/N)\\
&\leq   \mathbb{P}(\widetilde{W}_{k,N}\leq 1/N)+ N \mathbb{E}[W_{k,N}-\widetilde{W}_{k,N}] \leq 4 C e^{-(\log N)^2/C}. 
}
\end{proof}

\section{Upper bound}
Our aim of this section is to prove that for any $\varepsilon\in (0,1)$, there exists $M_1>0$ such that for any $M\geq M_1$, 
\aln{\label{Eq: goal for upper bound}
&\lim_{N\to\infty}\mathbb{P}\Big(\exists x\in \Iintv{-\sqrt{N},\sqrt{N}}^2,\,\\
&\qquad \qquad \qquad \sum_{k=1}^M \log W_{k,N}(x) > {}\frac{\sqrt 2(1+\varepsilon)}{{\sqrt M}}\sum_{i=1}^M  \lambda_{\tfrac{k-1}{M},\tfrac{k}{M}} \sqrt{\log N} +\varepsilon \sqrt{\log N} \Big)=0,\notag
}
which provides the upper bound.

We fix $M_0:= ({\sqrt 2}\sup_{t\in[0,1]}\sigma(t)+1)^2$ and we choose $M>M_0$ an integer, which eventually will be taken large as function of $\varepsilon$.
\subsection{Barrier argument}
Given $k\in\Iintv{1,M}$, $x\in \mathbb{Z}^2$, $s\in [0,2]$, and $\mathbf{\alpha}=(\alpha_\ell)_{\ell\geq k}$, we consider the sets
\begin{equation}
    \label{eq-Aell}
\mathcal{A}^{\mathbf{\alpha}}_{k,M}(x):= \{\forall \ell\in \Iintv{k,M},\, \log_+ W_{\ell,N}(x)\geq  \alpha_\ell\},\end{equation}
and
\aln{
\nonumber
{\rm Barrier}^{\varepsilon}_{k,M}&:=\Bigg\{\!(\alpha_i)_{i=k}^{M} \in \Iintv{0,M_0\sqrt{\log N}}^{M-k+1} :\\   \label{eq-Bar}
&\qquad 
\begin{array}{l}
\sum_{i=k}^M \alpha_i > \frac{\sqrt{2}(1+\varepsilon)}{\sqrt{M}}\sum_{i=k}^M  \lambda_{\tfrac{i-1}{M},\tfrac{i}{M}} \sqrt{\log N}  + \varepsilon \sqrt{\log N},\\
\sum_{i=\ell}^M \alpha_i \leq \frac{\sqrt{2}(1+\varepsilon)}{\sqrt{M}}\sum_{i=\ell}^M  \lambda_{\tfrac{i-1}{M},\tfrac{i}{M}}
\sqrt{\log N} + \varepsilon \sqrt{\log N},\,\forall \ell \geq k+1
\end{array}
\Bigg\}.
}
 Given $\ell \in \Iintv{1,M}$,  we write 
\begin{equation}\label{eq:defWbar}
\overline{W}_{\ell,N}:=\max_{x\in [0,{r_{\ell-2}})^2} W_{\ell,N}(x),\quad  \overline{\mu}_{t_{\ell-1}}(z) := \overline{\mu}_{t_{\ell-1},N}(z):=\max_{x\in [0,{r_{\ell-2}})^2} \mu_{t_{\ell-1},N}(x,z).
\end{equation}
\begin{prop}\label{prop: first goal for upper bound}
Assume $\varepsilon \in (0,1)$.   For any $M\in \N$, and for any   $N$ large enough depending on $\varepsilon,M$, then 
    \al{
    &\mathbb{P}\Big(\exists x \in\Lambda_{\sqrt N}, \sum_{k=1}^M \log W_{k,N}(x) > \tfrac{\sqrt 2 (1+\varepsilon)}{\sqrt{M}}\sum_{i=1}^M  \lambda_{\tfrac{i-1}{M},\tfrac{i}{M}}
    \sqrt{\log N}
    +\!2\varepsilon \sqrt{\log N}
    \Big)\\
    &\leq N^{\frac{3}{M}}{\max_{k\in\llbracket 1,M\rrbracket}} \sup_{\alpha\in \mathscr{B}_k}
    \Bigl\{ N^{ (M-k)/M } \mathbb{P}\Big(\forall \ell\geq k,\,\log_+  \overline{W}_{\ell,N}\geq  \alpha_\ell\Big)
   \Bigr \},
    }
    where we set  
    {$\mathscr{B}_k:={\rm Barrier}^{\varepsilon}_{k,M}$}.
\end{prop}
\begin{proof}
We estimate 
\al{
&\mathbb{P}\Big(\exists x\in \Lambda_{\sqrt{N}},
\sum_{i=1}^M \log W_{i,N}(x)\! > \!\tfrac{\sqrt 2 (1+\varepsilon)}{\sqrt{M}}\sum_{i=1}^M  \lambda_{\tfrac{i-1}{M},\tfrac{i}{M}}
\sqrt{\log N} + 2\varepsilon \sqrt{\log N}\Big)\\
&\leq \mathbb{P}(\exists x\in \Lambda_{\sqrt{N}},
\sum_{i=1}^M \lfloor \log_+ W_{i,N}(x)\rfloor > \tfrac{\sqrt 2 (1+\varepsilon)}{\sqrt{M}}\sum_{i=1}^M  \lambda_{\tfrac{i-1}{M},\tfrac{i}{M}}
\sqrt{\log N} + \varepsilon \sqrt{\log N})\\
&\leq \sum_{k\in\Iintv{1,M}} \mathbb{P}\bigg(\exists x\in \Lambda_{\sqrt{N}},\\
& \,\begin{array}{l} 
\sum_{i=k}^M\lfloor \log_+ W_{i,N}(x)\rfloor {>} \tfrac{\sqrt 2 (1+\varepsilon)}{\sqrt{M}}\sum_{i=k}^M  \lambda_{(i-1)/M,i/M} \sqrt{\log N} + \varepsilon \sqrt{\log N},\\
\sum_{i=\ell}^M\lfloor  \log_+ W_{i,N}(x)\rfloor {\leq} \tfrac{\sqrt 2 (1+\varepsilon)}{\sqrt{M}}\sum_{i=\ell}^M 
\lambda_{(i-1)/M,i/M}\sqrt{\log N} + \varepsilon \sqrt{\log N},\forall \ell\geq \! k+1\!\bigg).
\end{array}
}
The last probability can be further bounded from above by 
\al{
& \sum_{\alpha\in \mathscr{B}_k}
\!\mathbb{P}\left(\exists x \in \Lambda_{\sqrt{N}},
\forall \ell> k,\!\lfloor \log_+  W_{\ell,N} (x)\rfloor\!= \!\alpha_\ell,  \lfloor \log_+  W_{k,N} (x)\rfloor \geq  \alpha_k\right)\\
 &\leq \sum_{\alpha\in \mathscr{B}_k}
\! \mathbb{P}\left(\exists x\in \Lambda_{\sqrt{N}},
\forall \ell> k, \log_+  W_{\ell,N} (x)\geq  \alpha_\ell,\log_+  W_{k,N} (x) \geq  \alpha_k\right)\\
&\leq  (M_0\log N)^M \sup_{\alpha\in\mathscr{B}_k}
\mathbb{P}\left(\exists x\in \Lambda_{\sqrt{N}},
\mathcal{A}^{\mathbf{\alpha}}_{k,M}(x)\right),
}
where  we have used the union bound for $$\text{$\{\alpha_\ell=\lfloor\log_+ W_{\ell,N}(x) \rfloor\}$, $\ell \geq k+1$, and $\alpha_k = \lfloor \min\{\log_+ W_{k,N}(x),M_0\sqrt{\log N}\}\rfloor$}$$ 
 in the first line, and we denote $\mathbf{\alpha}=(\alpha_{\ell})_{\ell=k}^M$ in the last line. Moreover, we bound the last probability (recall $r_k = \lceil  N^{k/2M} \rceil$ that is close to $\sqrt{t_k}$) as:
\begin{align}
&\mathbb{P}(\exists x\in \Lambda_{\sqrt{N}},
\mathcal{A}^{\mathbf{\alpha}}_{k,M}(x))\nonumber\\
&\leq \sum_{\mathbf{k} \in [-2\sqrt{N},2\sqrt{N}]^2 \cap {r_{k-2}}\mathbb{Z}^2}\mathbb{P}(\exists x\in (\mathbf{k}+[0,{r_{k-2}})^2)\cap \mathbb{Z}^2,\,\mathcal{A}^{\mathbf{\alpha}}_{k,M}(x))\nonumber\\
&\leq (4\sqrt{N}+1)^2 r_{k-2}^{-2} \mathbb{P}(\exists x\in [0,{r_{k-2}})^2\cap \mathbb{Z}^2,\,\mathcal{A}^{\mathbf{\alpha}}_{k,M}(x))\nonumber\\
&\leq 32 N^{ (M-k+2)/M } \mathbb{P}\Big(\forall \ell\geq k,\,\max_{x\in [0,{r_{\ell-2}})^2} \log_+ W_{\ell,N}(x)\geq  \alpha_\ell\Big). \label{eq:probaForMarkov}
\end{align}
\end{proof}
\subsection{Decoupling}\label{section: decoupling}
Let $\widehat{\delta}>0$ be  a small constant chosen later depending on $\varepsilon$.    We fix $M$ and $L$  be large integers chosen later depending on $\varepsilon,\hat{\delta}$.  We consider a probability density function $\widehat{p}_n(x,y)$ such that
    \begin{equation}
            \widehat{p}_n(y) := \frac{1}{C_{\widehat{p}}(n)} \exp{\left(-\frac{\widehat{\delta} \|y\|^2}{ n}\right)},
        \end{equation}
    where $C_{\widehat{p}}(n)$ is a normalizing constant such that $\sum_{y\in \mathbb{Z}^2} \widehat{p}_n(y)=1$ and $C_{\widehat{p}}(n)=\Omega((\widehat{\delta} n)^{-1})$.  
    Assume $(\alpha_\ell)_{\ell=k}^M$ $\in$ $ {\rm Barrier}^{\varepsilon}_{k,M}$. 
We define 
\aln{\label{Eq: max W}
\mathcal{E}_{\ell,N}:= \{\forall i\in \Iintv{1,\ell-1},\,\forall x\in [0,r_{i-1})^2,\,\forall y\in  \mathbb{Z}^2,\,\mu_{t_{i}}(x,y)\leq N^{1/(L M^2)} \widehat{p}_{t_{i}}(y)\}.
}
\begin{lem}\label{lem: decay of kE}
  There exists $\widehat{\delta}>0$ such that   for any $M,L,A\in\N$,
if $N$ is large enough, then we have 
\begin{align*}
\mathbb{P}(\mathcal{E}_{M,N}^c)  \leq N^{-A}.
\end{align*}
\end{lem}
\begin{proof}
We note that, by the union bound,
    \al{
    &\mathbb{P}(\mathcal{E}_{M,N}^c)\leq \sum_{i=1}^M \sum_{x\in [0,r_{i-1})^2\cap \mathbb{Z}^2}\sum_{y\in \mathbb{Z}^2}\mathbb{P}(\mu_{t_{i}}(x,y)> N^{1/(LM^2)} \widehat{p}_{t_{i}}(y)).
    }
Note that for {$y\notin \Lambda_{t_i}(x)$}, the last probability is zero. 
We suppose {$y\in \Lambda_{t_i}(x)$}. By the lower tail concentration \cite[Proposition 3.1]{CaSuZy18}, we have
\al{
&\mathbb{P}(\mu_{t_{i}}(x,y)> N^{1/(LM^2)} \widehat{p}_{t_{i}}(y))\\&\leq \mathbb{P}(Z_{t_{i}}(x)\leq N^{-1/(2LM^2)})+\mathbb{P}(Z_{t_{i}}(x,y) p_{t_{i}}(x,y)> N^{1/(2LM^2)} \widehat{p}_{t_{i}}(y))\\
&\leq C e^{-(\log N)^2/C} +\mathbb{P}(Z_{t_{i}}(x,y) p_{t_{i}}(x,y)> N^{1/(2LM^2)} \widehat{p}_{t_{i}}(y)),
}
 with some $C =C(M,L,\hat{\beta})>0$. For the second term on the right-hand side, by the Markov inequality, for any $p\in \N$, we estimate
 \al{
&\mathbb{P}(Z_{t_{i}}(x,y) p_{t_{i}}(x,y)> N^{1/(2LM^2)} \widehat{p}_{t_{i}}(y))\\
&\leq  N^{-p/(2LM^2)}\widehat{p}_{t_{i}}(y)^{-p} \mathbb{E}[(Z_{t_{i}}(x,y) p_{t_{i}}(x,y))^p].
}
By Corollary~\ref{cor: LLT moment} and the local central limit theorem for the simple random walk  \cite[Theorem~2.3.11]{LL10}, with some $\delta\in (0,1)$, this is further bounded from above by 
\al{
&C' N^{-p/(2LM^2)} t_i^p\,  e^{p \hat{\delta}\|y\|^2/t_i^2}\, t_i^{-p/(1+\delta)}\, p_{t_{i}}(x,y)^{p\delta/(1+\delta)}\\
&\leq C'' N^{-p/(2LM^2)} e^{p \hat{\delta}\|y\|^2/t_i} e^{-p\delta\|x-y\|^2/(4t_i)} ,
} with some $C',C''>0$ depending on $\widehat{\delta},p$. 
For $\widehat{\delta}=\delta/8$, 
  if we take $p= \lceil 4A LM^2\rceil$, for $N$ large enough,  
   this is further bounded from above by $N^{-A}$ uniformly in  $x\in [0,r_{i-1})^2\cap \mathbb{Z}^2$ and $y\in \Lambda_{t_i}(x)$. 
    \end{proof}
    The next lemma gives an expression for conditional exponential moments of  $\overline{W}_{\ell,N}$, given $\mathcal{F}_{\ell-1}$. It creates the necessary decoupling between ratios of partition functions.
\begin{lem}\label{lem: a.s. L business}
For any $L\in\N$, $\alpha_\ell\in \N\cup\{0\}$, and  $q_\ell\in \N$,   it holds, almost surely,
    \al{
    &\mathbb{E}[\mathbf{1}_{\{\log_+{\overline{W}_{\ell,N}}\geq \alpha_\ell\}}\mathbf{1}_{\mathcal{E}_{\ell,N}}|~\mathcal{F}_{\ell-1}\Big]
    \leq e^{-q_\ell \alpha_\ell}  \mathbb{E}\Big[(\overline{W}_{\ell,N})^{q_\ell}\Big|~\kF_{\ell-1}\Big] \mathbf{1}_{\mathcal{E}_{\ell,N}} \\
    & \leq N^{M^{-2}} e^{-q_\ell \alpha_\ell}  \sum_{z_1,\ldots,z_{L}\in \mathbb{Z}^2} \widehat{p}_{t_{\ell-1}}(z_1)\cdots \widehat{p}_{t_{\ell-1}}(z_{L}) \mathcal{M}^L_{t_{\ell},q_\ell}(z_1,\ldots,z_{L}),
    }
where  
\(\mathcal{M}^L_{n,q}(z_1,\ldots,z_L):=\mathbb{E}\Big[Z_{n}(z_1)^{q/L} \cdots Z_{n}(z_L)^{q/L}\Big]
\).

\end{lem}
\begin{proof}
For any $L\in \N$,  since $\mathcal{E}_{\ell,N}$ is measurable with respect to $\kF_{\ell-1}$, we estimate 
\al{
&\mathbb{E}[ \mathbf{1}\{\log_+{\overline{W}_{\ell,N}}\geq \alpha_{\ell}\}
\mathbf{1}_{\mathcal{E}_{\ell,N}}|~\kF_{\ell-1}]
\leq e^{-q_\ell \alpha_\ell } \mathbb{E}\Big[(\overline{W}_{\ell,N})^{q_\ell}\Big|~\kF_{\ell-1}\Big]\mathbf{1}_{\mathcal{E}_{\ell,N}}\\
&\leq e^{-q_\ell \alpha_\ell } \mathbb{E}\Big[\Big(\sum_{z\in \mathbb{Z}^2}\overline{\mu}_{t_{\ell-1}}(z) Z_{t_{\ell-1},t_\ell}(z)^{q_\ell/L}\Big)^{L}\Big|~\kF_{\ell-1}\Big]\mathbf{1}_{\mathcal{E}_{\ell,N}},
}
where we have used Jensen's inequality in the last line. On the event $\mathcal{E}_{\ell,N}$, this is further equal to
\al{
& e^{-q_\ell \alpha_\ell }\sum_{z_1,\ldots,z_{L}\in \mathbb{Z}^2 }\overline{\mu}_{t_{\ell-1}}(z_1)\cdots \overline{\mu}_{t_{\ell-1}}(z_{L}) \mathbb{E}\Big[Z_{t_\ell-t_{\ell-1}}(z_1)^{q_\ell/L} \cdots Z_{t_\ell-t_{\ell-1}}(z_{L})^{q_\ell/L}\Big] \\
&\leq e^{-q_\ell \alpha_\ell } N^{M^{-2}} \sum_{z_1,\ldots,z_{L}\in \mathbb{Z}^2} \widehat{p}_{t_{\ell-1}}(z_1)\cdots \widehat{p}_{t_{\ell-1}}(z_{L}) \mathbb{E}\Big[Z_{t_\ell}(z_1)^{q_\ell/L} \cdots Z_{t_\ell}(z_{L})^{q_\ell/L}\Big].
}
\end{proof}
\subsection{Diffusive scale separation}
Next, we estimate 
$$\sum \widehat{p}_{t_{\ell-1}}(z_1)\cdots \widehat{p}_{t_{\ell-1}}(z_{L}) \mathcal{M}^{L}_{t_\ell,q_\ell}(z_1,\ldots,z_{L}).$$
Our goal in this section is to prove that we are able to restrict the sum over $(z_i)_{i=1}^L$ satisfying $\|z_i-z_j\|\geq (r_{\ell-1})^{1-\varepsilon}$ for arbitrary small $\varepsilon>0$. We first prove the following:
\begin{lem}\label{lem: diffusive scaling}
    For any $M\in\N$ and $\varepsilon_0,A>0$, {if $L$ is large enough}, then for $N\in\N$  large enough, for any $\ell \in \Iintv{1,M-1}$, we have 
\aln{
\widehat{\rm P}^{\otimes L}_0(\exists I\subset \Iintv{1,L},|I|\geq (1-\varepsilon_0)L,\,\forall a\neq b\in I,\,\|S_{t_{\ell}}^a-S_{t_{\ell}}^b\|\geq (r_{\ell})^{1-\varepsilon_0})\geq 1-N^{-A},
}
where $(S_{t}^a)_{a\in \Iintv{1,L}}$ 
are i.i.d.\@ and distributed according to  $\widehat{p}_{t}(\cdot)$.
\end{lem}
\begin{proof}
{Let $L\in \N$.} We prove the lemma by contradiction: suppose that for any $I\subset \Iintv{1,L}$ with $|I|\geq (1-\varepsilon_0)L,$ there exist $a\neq b\in I$ such that $\|S_{t_{\ell}}^a-S_{t_{\ell}}^b\|< (r_{\ell})^{1-\varepsilon_0}$.  We define $$I_*:=\{b\in \Iintv{1,L}|~\forall  a\in \Iintv{1,b-1},\,\|S_{t_{\ell}}^a-S_{t_{\ell}}^b\|\geq (r_{\ell})^{(1-\varepsilon_0)}\}.$$
Since $\|S_{t_{\ell}}^a-S_{t_{\ell}}^b\|\geq (r_{\ell})^{1-\varepsilon_0}$ for any $a<b\in I_*$, by the assumption, we have $|I_*|<(1-\varepsilon_0)L$, in particular $|J_*|\geq \varepsilon_0 L$ with $J_*:=\Iintv{1,L}\setminus I_*$.  This implication, together with the union bound, yields
\aln{
& \widehat{\rm P}^{\otimes L}_0(\forall I\subset \Iintv{1,L},\,|I|\geq (1-\varepsilon_0)L,\,\exists  a\neq b\in I,\,\|S_{t_{\ell}}^a-S_{t_{\ell}}^b\|\!< \!(r_{\ell})^{1-\varepsilon_0})\notag\\
 & \leq \sum_{J\subset \Iintv{1,L},|J|\geq \varepsilon_0 L} \widehat{\rm P}^{\otimes L}_0(\forall b\in J,\,\exists a\in\Iintv{1,b-1},\,\|S_{t_{\ell}}^a-S_{t_{\ell}}^b\|<(r_{\ell})^{1-\varepsilon_0} ) \notag \\
 &=  \sum_{J=(j_1<\ldots<j_{k})\subset \Iintv{1,L},k\geq \varepsilon_0 L} \hat{{\rm E}}^{\otimes L}_0\Big[\prod_{s=1}^{k} \mathbf{1}\{\exists a\in\Iintv{1,j_s -1},\,\|S_{t_{\ell}}^a-S_{t_{\ell}}^{j_s}\|<(r_{\ell})^{1-\varepsilon_0}\} \Big].\label{eq: separate diffusivity}
 }
Let $c=c(\varepsilon_0,M)>0$ be such that for any $k\in \Iintv{1,M-1}$, $\max_{z\in\mathbb{Z}^2}\widehat{\rm P}_0(\|S_{t_{k}}-z\|< (r_{k})^{1-\varepsilon_0})\leq N^{-c}$. We define $\mathcal{E}_b:=\{\exists a\in\Iintv{1,b-1},\,\|S_{t_{\ell}}^a-S_{t_{\ell}}^{b}\|<(r_{\ell})^{1-\varepsilon_0}\}$. Note that
\al{
 \widehat{\rm P}^{\otimes b}(\mathcal{E}_b|~S^1,\ldots,S^{b-1})&= \widehat{\rm P}^{\otimes b}(\exists a\!\in\!\Iintv{1,b-1},\|S_{t_{\ell}}^a-S_{t_{\ell}}^b\|\!<\!(r_{\ell})^{1-\varepsilon_0}|S^1,\ldots,S^{b-1})\\
&\leq L \max_{z\in\mathbb{Z}^2}\widehat{\rm P}(\|S_{t_{\ell}}-z\|< (r_{\ell})^{1-\varepsilon_0})\leq L N^{-c}\quad\text{a.s.}
}
Assume $ L \geq  (1+A)/(\varepsilon_0 c)$. If $N$ is large enough depending on $L,$   by Lemma~\ref{lem: diffusive scaling}, then \eqref{eq: separate diffusivity} can be  further bounded from above inductively by
\al{
& \sum_{J=(j_s)_{s=1}^{k}\subset \Iintv{1,L},k\geq \varepsilon_0 L} \widehat{\rm E}^{\otimes L}_0\Big[\prod_{s=1}^{k-1} \mathbf{1}_{\mathcal{E}_{j_s}} \widehat{\rm P}_0^{\otimes j_{k}}( \mathcal{E}_{j_{k}} |~S^1,\ldots,S^{j_{k}-1})\Big]\\
 &\leq \sum_{J=(j_s)_{s=1}^{k}\subset \Iintv{1,L},k\geq \varepsilon_0 L} LN^{-c}\widehat{\rm E}^{\otimes L}_0\Big[\prod_{s=1}^{k-1} \mathbf{1}_{\mathcal{E}_{j_s}} \Big]\\
 &\leq  \sum_{J=(j_s)_{s=1}^{k}\subset \Iintv{1,L},k\geq \varepsilon_0 L} L^L N^{-ck}\leq L^{L} 2^L N^{-c\varepsilon_0 L}\leq N^{-A},
 }
where  $2^L$ accounts for the choice of $J$.
 \end{proof}
 \begin{cor}\label{cor: diffusive separate}
Consider sequences $q_k = q_k(N)$ with $k\in \Iintv{2,M}$ satisfying \eqref{eq:conditionqk}. For any $M\in \N$ large enough and $\varepsilon_0>0$ small enough so that $q_k' := q_k/(1-\sqrt{\varepsilon_0})$ satisfies \eqref{eq:conditionqk}, if $L\in \N$ is large enough, then  for any $k\in \Iintv{2,M}$,    we have
     \al{
&\sum_{z_1,\ldots,z_{L}\in \mathbb{Z}^2} \widehat{p}_{t_{k-1}}(z_1)\cdots \widehat{p}_{t_{k-1}}(z_{L}) \mathcal{M}^{L}_{t_{k},q_\ell}(z_1,\ldots,z_{L}) \\
&\;\;\;\;\leq  N^{\varepsilon_0} \sup_{z_1,\ldots,z_L,\|z_i-z_j\|\geq (r_{k-1})^{1-\varepsilon_0}}  \E \left[{\prod_{i=1}^L} Z_{t_{k}}(z_i)^{q_k'/L}\right].
}
 \end{cor}
 \begin{proof}
Note that if there exists $I\subset \Iintv{1,L}$ such that $|I|\geq (1-\varepsilon_0)L$ and for all $a\neq b\in I,\,\|z_a-z_b\|\geq (r_{k-1})^{1-\varepsilon_0},$ then we have with $q'_k=q_k/(1-\sqrt{\varepsilon_0})$,
\al{\mathcal{M}^L_{t_{k},q_{k}}(z_1,\ldots,z_L)
&=\mathbb{E}\Big[\prod_{i\in I} Z_{t_{k}}(z_i)^{q_k/L} \prod_{j\notin I}Z_{t_{k}}(z_j)^{q_k/L}\Big]\\
&\leq \E \Big[\prod_{i\in I} Z_{t_{k}}(z_i)^{q'_k/L}\Big]^{1-\sqrt{\varepsilon_0}} \E \Big[\prod_{j\notin I} Z_{t_{k}}(z_j)^{q_k/(L\sqrt{\varepsilon_0})}\Big]^{\sqrt{\varepsilon_0}}\\
&\leq \E \Big[\prod_{i\in I} Z_{t_{k}}(z_i)^{q'_k/L}\Big]^{1-\sqrt{\varepsilon_0}} \E \Big[ Z_{t_{k}}^{q_k\sqrt{\varepsilon_0}}\Big]^{\sqrt{\varepsilon_0}},
}
where {in the first inequality,}  we have used {H\"older's inequality, and} in the last line, the 
generalized H\"older inequality $\mathbb{E}[|\prod_{i=1}^n X_i|]\leq\prod_{i=1}^n \mathbb{E}[| X_i|^n]^{1/n}$ {assuming $L{\varepsilon_0}>1$}. For $\varepsilon_0$ small enough, the second term on the right-most hand side is negligible, i.e., $\E \Big[ Z_{t_{k}}^{q_k\sqrt{\varepsilon_0}}\Big]=N^{\mathcal O(\varepsilon_0)}$ by Theorem~\ref{thm:qmomentupperBound}.  By the same theorem, 
if we take $A$ large enough, then
\al{
\sup_{z_1,\ldots,z_L}
\mathcal{M}^L_{t_{k},q_{k}}(z_1,\ldots,z_L)
\leq \mathbb{E}[Z_{t_{k}}^{q_k}]
\leq N^{A},
}
where we have used the generalized H\"older inequality in the first inequality as above. Putting things together with Lemma~\ref{lem: diffusive scaling}, we obtain that for any $\varepsilon_0>0$ small enough and $A>0$ large enough, if we take $L$ large enough depending on $\varepsilon_0,A$, then for $N$ large enough, 
    \aln{
&\sum_{z_1,\ldots,z_{L}\in \mathbb{Z}^2} \widehat{p}_{t_{k-1}}(z_1)\cdots \widehat{p}_{t_{k-1}}(z_{L}) \mathcal{M}^{L}_{t_{k},q_k}(z_1,\ldots,z_{L})\notag\\
&\leq N^{\varepsilon_0/2}\, 2^{L} \sup_{\substack{I\subset \Iintv{1,L},\,(z_i)_{i\in I} \\ \forall i \neq j \in I, \|z_i - z_j\| \geq (r_{k-1})^{1-\varepsilon_0}}}  \E \Big[\prod_{i\in I} Z_{t_{k}}(z_i)^{q_k'/L}\Big]^{1-\varepsilon_0} + 1\notag\\
&\leq  N^{\varepsilon_0} \sup_{(z_i)_{i=1}^L,\|z_i-z_j\|\geq (r_{k-1})^{1-\varepsilon_0}}  \E \Big[ {\prod_{i=1}^L} Z_{t_{k}}(z_i)^{q_k'/L}\Big],\label{Eq: second goal}
}
where the factor $2^L$  accounts for the  choices of $I$.
\end{proof}
It remains to estimate the last expectation of Corollary~\ref{cor: diffusive separate}. 

\subsection{Moment estimates for far away starting packs.}
For all $I\subset \Iintv{1,q}$, define $\mathcal C_I := \{(i,j):~i,j\in I,\,i<j\}$ and $\mathcal C_q=\mathcal C_{\Iintv{1,q}}.$
Define: 
\begin{equation} \label{eq:defpsi}
\psi_{s,t,I} := \beta_N^2\sum_{n=s}^{t} \sum_{(i,j)\in \mathcal C_I} \mathbf{1}_{S_n^i=S_n^j},
\end{equation}
and set $\psi_{n,I}:=\psi_{1,n,I}$.
We use the short-notation $\psi_{s,t,q}$ and $\psi_{t,q}$ when $I=\Iintv{1,q}$. 
We note that
\begin{equation} \label{eq:momentFormula}
\E\left[\prod_{i\in I} Z_{s,t}(x_i,\beta_N)\right] =  {\rm E}_{(x_i)_{i\in I}}^{\otimes |I|}\left[e^{\beta_N^2 \psi_{s,t,I}}\right].
\end{equation}
Recall the definition of $\lambda^2_{a,b}$ in \eqref{eq:def_lambda_T_Nbis}. The goal of this section is to prove the following.
\begin{lem}\label{lem: L business bound}
Consider sequences $q_k = q_k(N)$ with $k\in \Iintv{1,M}$ satisfying \eqref{eq:conditionqk}. For any $M\in\N$, there exists $C=C(\hat{\beta},\delta,M)>0$ such that,  if $\varepsilon_0$ is small enough and $L\in\N$ is large enough depending on $M$, then for any  $k\in\Iintv{1,M},$   and $z_u\in \mathbb Z^2$ with $\|z_u-z_v\|> (r_{k-1})^{1-\varepsilon_0}$ for all $u\neq v \in \Iintv{1,L}$, 
\begin{equation} \label{eq:ZqL}
\E \Big[ \prod_{u=1}^{L} Z_{t_k}(z_u)^{q_k/L}\Big]\leq C \exp{\left(\lambda^2_{(k-1)/M,k/M} \binom {q_k} 2 + \frac{\log N}{M^2} \right)}.
\end{equation}
\end{lem}  
\begin{proof}
We assume that ${q_k}/{L}$ is an integer; if not, we replace $q_k$ with $L \lceil q_k / L \rceil$ to ensure divisibility, which does not affect the argument. Moreover, the claim follows directly from Theorem~\ref{thm:qmomentupperBound} when $k = 1$. Hence, we assume $k \in \Iintv{2,M}$ for the remainder of the proof.

Define the \emph{packs} of indices: $\mathcal Q_u := \llbracket ((u-1)q_k/L)+1,uq_k/L\rrbracket$ for $u\in \Iintv{1,L}$. We have by \eqref{eq:momentFormula}:
\[
\E \left[\prod_{i=1}^{L} Z_{t_k}(z_i)^{q_k/L}\right] = \rme^{\otimes q_k}_{z_1,\dots,z_{L}}\left[\exp(\psi_{t_k,q_k})\right], 
\]
where in the second expectation, the walks $(S^j,{j\in \mathcal Q_u})$  are all started at the same point $z_u\in \mathbb Z^2$ within each pack $u\in \Iintv{1,L}$. Now,  introduce the event that two walks from different packs meet before time $t_{k-1}':=(t_{k-1})^{1-2\varepsilon_0}$:
\[
\mathcal{E}'_k:= \left\{ \exists u\neq v \in \Iintv{1,L}, \exists i\in Q_u, \exists  j \in \mathcal Q_v, \exists  n\leq t_{k-1}' ,\, S_n^i= S_n^j \right\}.
\]
By moderate deviation estimates for simple random walks \cite[Theorem~2.3.11]{LL10} and the union bound, we have
\[\rmP^{\otimes q_k}_{z_1,\dots,z_{L}}(\mathcal{E}'_k) \leq Cq_k^2 N e^{-N^{\varepsilon_0/M}} \leq C e^{-N^{\varepsilon_0/M}/C}
\]
with some $C=C(\varepsilon_0,M)>1$. Hence, by H\"older’s inequality with $a^{-1}+b^{-1} = 1$ satisfying $a,b>1$, $\hat \beta \sqrt a < 1$, and \eqref{eq:conditionqk} holding with $a\beta_N^2$ in place of $\beta_N^2$, we have
\begin{equation} \label{eq:HolderProba}
\rme^{\otimes q_k}_{z_1,\dots,z_{L}}\left[\exp(\psi_{t_k,q_k}) \mathbf{1}_{\mathcal{E}'_k}\right] \leq C\rme^{\otimes q_k}_{z_1,\dots,z_{L}}\left[\exp(a\psi_{t_k,q_k})\right]^{1/a} e^{-\tfrac{N^{\varepsilon_0/M}}{Cb}} \leq C' e^{-\tfrac{N^{\varepsilon_0/M}}{C'}},
\end{equation}
with some $C'=C'(\hat \beta,\delta,M)>0$,
where we have used Theorem~\ref{thm:qmomentupperBound} which implies that $\sup_{X,M} \rme^{\otimes q_k}_{X}\left[\exp(a\psi_{t_k,q_k})\right] 
=N^{\mathcal O(1)}$. Now, let 
\[
    D_k:=\{X=(x_1,\dots,x_{q_k})\in (\mathbb Z^2)^{q_k} \!:\! \sharp \{(i,j)\in \mathcal C_{q_k} : \|x_i-x_j\|^2 \leq (t_{k-1})^{1-3\varepsilon_0}\}< p_0 \},
\]
and $\mathcal{E}_k'':=\{(S^i_{t_{k-1}'})_{i\in \Iintv{1,q_k}}\notin D_k\}$
where $p_0$ is a large but finite number chosen later. Suppose $\mathcal{E}_k''$ holds, i.e., there exist $(i_1,j_1),\ldots,(i_{p_0},j_{p_0})\in \kC_{q_k}$ such that $\|S_{t_{k-1}'}^i-S_{t_{k-1}'}^j\|^2 \leq (t_{k-1})^{1-3\varepsilon_0}$. We define 
$$I_*:=\{j\in \Iintv{1,q_k}:~\exists i \in [j-1],\,\|S_{t_{k-1}'}^i-S_{t_{k-1}'}^j\|^2\le  (t_{k-1})^{1-3\varepsilon_0} \}.$$
Since $\bigcup_{k=1}^{p_0} \{i_k,j_k\} \subset I_*$, we have $|I_*|\geq p_0$. Therefore, by the same argument as in Lemma~\ref{lem: diffusive scaling}, we conclude
\al{
&\sup_{x_1,\dots,x_{L} \in \mathbb Z^2} \rmP^{\otimes q_k}_{x_1,\dots,x_{L}}(\mathcal{E}_k'') \\
&\leq \!\!\sup_{x_1,\dots,x_{L} \in \mathbb Z^2}\! \sum_{I\subset \Iintv{1,q_k},|I|= p_0}\!\!\!\!\!\!\!\!\!\rmP^{\otimes q_k}_{x_1,\dots,x_{L}}(\forall j\in I, \exists i \in [j-1],\|S_{t_{k-1}'}^i-S_{t_{k-1}'}^j\|^2\!\le \! (t_{k-1})^{1-3\varepsilon_0} )\\
&\leq \binom{q_k}{p_0} {N^{-\varepsilon_0 p_0/M}}.
}
Similarly as in \eqref{eq:HolderProba}, we get that $
    \rmP^{\otimes q_k}_{z_1,\dots,z_{L}}\left[\exp(\psi_{N,q_k}) \mathbf{1}_{\mathcal{E}_k''}\right] \leq N^{-\frac{\epsilon_0  p_0}{C' M}+C'}$ for some $C'=C'(\hat \beta,\delta) > 0$.
Therefore, by choosing $p_0$ large enough depending on $M,C',\varepsilon_0$, we are reduced to the complementary event $(\mathcal{E}_k')^c\cap (\mathcal{E}_k'')^c$. 

On $(\mathcal{E}_k')^c$, we have $\exp(\psi_{t_k,q_k})\!=\!\exp(\psi_{t_{k-1}'+1,t_k,q_k})\prod_{u=1}^{L} \exp(\psi_{t_{k-1}'{+1},\mathcal Q_{u}})$. Hence, by Markov’s property 
we have
\begin{equation} \label{eq:keybound}
\begin{split}
        &\rme^{\otimes q_k}_{z_1,\dots,z_{L}}\left[\exp(\psi_{t_k,q_k}) \mathbf{1}_{(\mathcal{E}_k')^c\cap (\mathcal{E}_k'')^c}\right] \\&\leq \rme^{\otimes q_k}_{z_1,\dots,z_{L}}\left[\prod_{u=1}^{L} \exp(\psi_{t_{k-1}',\mathcal Q_{u}})  \right] \sup_{X\in D_k} \rme^{\otimes q_k}_X [\exp(\psi_{t_k-t_{k-1}',q_k})]\\
    &\leq \rme^{\otimes q_k}_{z_1,\dots,z_{L}}\left[\prod_{u=1}^{L} \exp(\psi_{t_{k-1}',\mathcal Q_{u}})  \right] \sup_{X\in D_k} \rme^{\otimes q_k}_X [\exp(\psi_{t_k,q_k})].
    \end{split}
\end{equation}
Now, by independence of the packs and Theorem~\ref{thm:qmomentupperBound}, we have
\begin{equation}
\begin{split}
\rme^{\otimes q_k}_{z_1,\dots,z_{L}}\left[\prod_{u=1}^{L} \exp(\psi_{t_{k-1}',\mathcal Q_{u}}) \right] 
&= \prod_{u=1}^{L} \rme^{\otimes (q_k/L)}_{z_u} \left[\exp(\psi_{t_{k-1}',\mathcal Q_{u}}) \right] \\
&\leq   e^{C'' L (q_k/L)^2},
\end{split}
\end{equation}
with some $C''=C''(\hat{\beta},\delta)>0$, where the bound holds uniformly in $z_1,\dots,z_{L}$.  It remains to control the second term of the right-hand side of \eqref{eq:keybound}.  For any $X=(x_1,\dots, x_{q_k})\in D_k$, observe that there are at most $2p_0$ particles that are at distance less than $(t_{k-1})^{(1-3\varepsilon_0)/2}$ from any other particle, that is,
\[
    \sharp \left\{i\leq q_k : \exists j\in \Iintv{1,q_k}\setminus\{i\}, \|x_i-x_j\|^2 \leq (t_{k-1})^{1-3\varepsilon_0}\right\} \leq 2p_0.
\]
Up to permuting indices, we can suppose that for any {$i\in \Iintv{2 p_0+1,q_k}$, $j\in \llbracket 1,q_k\rrbracket\setminus\{i\}$, we have} $ \|x_i-x_j\|^2 \geq (t_{k-1})^{1-3\varepsilon_0}$. By repeating the argument of \eqref{eq:HolderProba}, we can neglect the event
\[\mathcal{E}_k''':= \left\{ \exists i\in  \llbracket 2p_0+1,q_k\rrbracket, \exists j\in \Iintv{1,q_k}, \exists n\leq (t_{k-1})^{1-4\varepsilon_0},\, S_n^i= S_n^j \right\}.\]
Then on $(\mathcal{E}_k''')^c$, one has $\psi_{t_k,q_k} = \psi_{(t_{k-1})^{1-4\varepsilon_0},2p_0}+ \psi_{(t_{k-1})^{1-4\varepsilon_0}+1,t_k,q_k} $, so that by \eqref{eq:momentFormula}, for any $\varepsilon_1>0$ small enough,  for all $X = (x_i)_{i=1}^{q_k}\in D_k$, 
\begin{align*}
&\rme^{\otimes q_k}_X [\exp(\psi_{t_k,q_k})\mathbf{1}_{(\mathcal{E}_k''')^c}]  = \mathbb{E}\left[\left(\prod_{i=1}^{2 p_0} Z_{t_k}(x_i) \right)\left(\prod_{i=2 p_0+1}^{q_k} Z_{(t_{k-1})^{1-4\varepsilon_0},t_k}(x_i) \right)\right] \\
& \!\!\!\leq \mathbb{E}\left[\prod_{i=1}^{2 p_0} \left( Z_{t_k}(x_i) \right)^{(1+\varepsilon_1)/\varepsilon_1}\right]^{\varepsilon_1/(1+\varepsilon_1)} \mathbb{E}\left[\prod_{i=2 p_0+1}^{q_k} \left(Z_{(t_{k-1})^{1-4\varepsilon_0},t_k}(x_i) \right)^{1+\varepsilon_1}\right]^{1/(1+\varepsilon_1)}
 \\
 &\!\!\!\leq C(\varepsilon_1) \exp{\left(\lambda^2_{(1-4\varepsilon_0)(k-1)/M,k/M} (\hat{\beta})\left( \binom{q_k'} 2+\varepsilon_1 \log N \right)\right)},
\end{align*}
with {$q_k' = (1+\varepsilon_1)q_k$} and $C(\varepsilon_1)>0$, where we have used {the generalized} H\"older’s inequality and Theorem~\ref{thm:qmomentupperBound}.   The claim \eqref{eq:ZqL}
then holds by choosing $\varepsilon_0,\varepsilon_1$ small enough and $L$ large enough depending on $M$.
\end{proof}
The following lemma will be useful in collecting the estimates from this section.
\begin{lem}  \label{lem:combLem}
    Let $M>1$. Consider sequences $q_k = q_k(N)$ with $k\in \Iintv{1,M}$ satisfying \eqref{eq:conditionqk}. Let $\alpha=(\alpha_{\ell})_{\ell=k,\dots,M}$ with $\alpha_\ell \geq 0$. Then, for $N$ large enough, 
    a.s.
    \[
    \mathbb{P}(\log_+{\overline{W}_{k,N}}\geq \alpha_k,\mathcal{E}_{k,N}|~\mathcal{F}_{k-1})
    \leq N^{\frac{2}{M^2}} \exp{\left\{-q_k \alpha_k + \lambda^2_{(k-1)/M,k/M} \binom {q_k} 2 \right\}}.
    \]
Furthermore,  {for all $k\in \llbracket 1,M\rrbracket$,}
\[
\mathbb{P}\Big(\forall \ell \geq k,\,\log_+{\overline{W}_{\ell,N}}\geq \alpha_{\ell}\Big)
\leq N^{-2} \!+ \! N^{\frac 2 M} \exp\left\{-\sum_{\ell=k}^M q_\ell \alpha_\ell + \lambda^2_{(\ell-1)/M,\ell/M} \binom {q_\ell} 2\right\}.
\]
\end{lem}
\begin{proof}
Lemma~\ref{lem: a.s. L business},  Corollary~\ref{cor: diffusive separate}, and Lemma~\ref{lem: L business bound} yield the first claim.
We estimate
\begin{align*}
&\mathbb{P}\left(\forall \ell\geq k,\,\log_+{\overline{W}_{\ell,N}}\geq \alpha_{\ell}; \mathcal{E}_{M,N}\right)\\
&\leq  \mathbb{E}\left[\prod_{\ell = k}^{M-1} \mathbf{1}_{\{\log_+{\overline{W}_{\ell,N}}\geq \alpha_{\ell}\}} \mathbf{1}_{\mathcal{E}_{M-1,N}}\mathbb{P}( \log_+{\overline{W}_{M,N}}\geq \alpha_M;\mathcal{E}_{M,N}|~\kF_{M-1})\right],
\end{align*}
and apply the first claim iteratively.
The second claim follows by Lemma~\ref{lem: decay of kE}.
\end{proof}
\subsection{Proof of \eqref{Eq: goal for upper bound}} 
Fix $\delta \in (0,1/3)$ a real number that may depend on $\varepsilon$; the exact dependence is chosen below.
Define
\begin{equation}\label{def:q_khat} \widehat{q}_k := \frac{\sqrt{2}(1-\delta)\sqrt{\log N}}{\sqrt{M}\lambda_{(k-1)/M,k/M}}.
\end{equation}
By Proposition~\ref{prop: first goal for upper bound} and Lemma \ref{lem:combLem}, and  a union bound over $M$ events, we have
\begin{align}
  &\mathbb{P}\Big(\exists x\in \Lambda_{N^{1/2}}, \sum_{k=1}^M \log W_{k,N}(x) > \tfrac{\sqrt{2}(1+\varepsilon)}{\sqrt{M}}\sum_{i=1}^M  \lambda_{\tfrac{i-1}{M},\tfrac{i}{M}} \sqrt{\log N} + 2\varepsilon\sqrt{\log N}\Big)\nonumber\\
    &\leq N^{-1+\frac{2}{M}}+{M}N^{\frac{5}{M}}\nonumber\\
    & \times \!\! \! \sup_{k\in\Iintv{1,M}} \!\left\{ N^{\tfrac{M-k}{M}} \!\exp{\Big\{\Big(-\!\!\inf_{(\alpha_k,\ldots,\alpha_M)\in {\rm Barrier}^{\varepsilon}_{k,M}} \sum_{\ell=k}^M \widehat{q}_\ell \alpha_\ell\Big)\!+\!\frac12\sum_{\ell=k}^M  \lambda^2_{(\ell-1)/M,\ell/M} \widehat{q}_\ell^2\Big\}}\!\right\}. \label{eq-180225a}
\end{align}
Let
 $$ f(u) := \frac{\sqrt{2}(1+\varepsilon)  \lambda_{(u-1)/M,u/M}}{\sqrt{M}}\sqrt{\log N} = \frac{2(1+\varepsilon)(1-\delta) \log N}{M \widehat{q}_u}  ,$$ and $a:=  \varepsilon \sqrt{\log N}$, and write
 $$ \sum_{\ell=k}^M \widehat{q}_\ell \alpha_\ell=
 \frac{2(1+\varepsilon)(1-\delta)\log N}{M}\sum_{\ell=k}^M  \frac{\alpha_\ell}{f(\ell)}.$$
Applying Lemma~\ref{lem: variationalProblems} with $g(\ell)=\alpha_\ell$ to estimate
the infimum in \eqref{eq-180225a}, we get that
the supremum in \eqref{eq-180225a} can be further bounded by
\al{
&\sup_{k\in\Iintv{1,M}} \bigg\{ e^{-\widehat{q}_M \varepsilon \sqrt{\log N}}  N^{\tfrac{M-k}{M} } \\
&\qquad \times \exp{\Big( -2(1-\delta) (1+\varepsilon) \frac{M-k+1}{M}\log N + \frac{M-k+1}{M}(1-\delta)^2\log N\Big)}\bigg\},
}
which goes to $0$ as $N\to\infty$ if we take $M$ large enough and $\delta>0$ small enough  depending on $\varepsilon$. 
\subsection{A variant moment estimate for partition functions}
As a byproduct from the previous arguments, we have the following:
\begin{cor}\label{cor: VariantMomentEstimate}
    For any $\varepsilon_0>0$ and sequences $q_k = q_k(N)$ with $k\in \Iintv{1,M}$ satisfying \eqref{eq:conditionqk}, we have for all $N$ large,
    \al{
\mathbb{E}\Big[\prod_{\ell=1}^M ({W}_{\ell,N})^{q_\ell}\Big]\leq     N^{\varepsilon_0} \prod_{\ell=1}^M e^{\binom{q_\ell}{2} \lambda_{(\ell-1)/M,\ell/M}^2 }.
    }
\end{cor}
\begin{proof}
Assume $\varepsilon_0>0$. Let $\delta>0$ be small enough and satisfying $q_k'=(1+\delta)q_k\in \N$, which will be chosen later. By \eqref{eq:defpsi} and Markov's property, we have
\aln{\label{eq: variant moment estimate}
\mathbb{E}\left[\prod_{k=1}^{M} ({W}_{k,N})^{q_k'}\right] \leq \prod_{k=1}^{M} \sup_{\mathbf x} \rme_{\mathbf x}^{\otimes q_k'} \left[ {e^{\psi_{t_k-t_{k-1},q_k'}}}\right]
\leq N^{C_1},
}
for some $C_1=C_1(\hat \beta,\delta,M)$, by Theorem \ref{thm:qmomentupperBound}.  
 By Lemma~\ref{lem: decay of kE}, for any $A>0$, we have that for all $N$ large,
\al{
\mathbb{E}\left[\prod_{k=1}^{M} ({W}_{k,N})^{q_k} \mathbf{1}_{\kE_{M,N}^c}\right]&\leq \mathbb{E}\left[\prod_{k=1}^{M} {W}_{k,N}^{q_k'}\right]^{1/(1+\delta)} \mathbb{P}(\kE_{M,N}^c)^{\delta/(1+\delta)}\leq N^{-A}.
}
On the other hand, by Lemma~\ref{lem: a.s. L business} and Corollary~\ref{cor: diffusive separate},  
we have
    \al{
 & \mathbb{E}\left[\prod_{k=1}^{M} ({W}_{k,N})^{q_k} \mathbf{1}_{\kE_{M,N}}\right]\\
&\leq N^{-M^{-1}}  \prod_{\ell=1}^M \sum_{z_1,\ldots,z_{L}\in \mathbb{Z}^2} \widehat{p}_{t_{\ell-1}}(z_1)\cdots \widehat{p}_{t_{\ell-1}}(z_{L}) \mathcal{M}^{L}_{t_{\ell},q_\ell}(z_1,\ldots,z_{L})\notag\\
&\leq  N^{\varepsilon_0} \prod_{\ell=1}^M \sup_{z_1,\ldots,z_L,\|z_i-z_j\|\geq (r_{\ell-1})^{1-\varepsilon_0}}  \E \Big[\prod_{i=1}^L Z_{t_{\ell}}(z_i)^{q_\ell'(1+\varepsilon_0)/L}\Big].
}
Together with Lemma~\ref{lem: L business bound}, the proof is completed.
\end{proof}
\section{Lower bound}
We fix $\delta,\varepsilon_1\in (0,1/3)$, and $L,M\in\N$, which are chosen later. For simplicity of notation, we assume $(1-\varepsilon_1)M\in \N.$ Set 
$$\alpha_\ell:= \sqrt{2}(1-2\delta) \lambda_{(k-1)/M,k/M}\sqrt{\log N}/\sqrt{M}, \quad \alpha_\ell':=\sqrt{2}  \lambda_{(k-1)/M,k/M}\sqrt{\log N}/\sqrt{M}.$$ 
Recall $\widetilde{W}_{k,N}(x)$ in \eqref{eq-tildeW}. We define $$\widetilde{\mathcal{A}}_{M}(x):=\{\log \widetilde{W}_{\ell,N}(x)\geq \alpha_\ell,\,\forall \ell\in \Iintv{1,(1-\varepsilon_1)M}\}.$$ 
We first show the following:
\begin{lem}\label{lem: key lemma for lower bound}
    For any $\varepsilon,\varepsilon_1>0$, if $M$ is large enough and $\delta$ is small enough, then for any $N\in \N$ large enough, for any  $k\in\Iintv{1,(1-\varepsilon_1)M}$, and  $x,y\in \Lambda_{\sqrt{N}}$ with $\|x-y\| \geq  r_{k}$,
\aln{
& \mathbb{P}(\widetilde{\mathcal{A}}_{M}(x)\cap \widetilde{\mathcal{A}}_{M}(y))\leq N^{-1-\frac{k-1}{M}+\varepsilon_1+\varepsilon},\label{eq:upperBouondForIntersection}\\
& \mathbb{P}(\widetilde{\mathcal{A}}_{M}(x))\geq  N^{-1+\varepsilon_1-\varepsilon}. \label{Eq: lower bound for single}
}
\end{lem}

\begin{proof}[Proof of \eqref{eq:upperBouondForIntersection}]
Recall that $\mathbb{P}(\mathcal{E}_{M,N}^c)$ decays faster than any polynomial rate by Lemma~\ref{lem: decay of kE}. We have that $\mathbb{P}(\widetilde{\mathcal{A}}_{M}(x)\cap \widetilde{\mathcal{A}}_{M}(y)\cap\mathcal{E}_{M,N}^c)$ is smaller than
\[\mathbb{P}\Big(\cap_{\ell=1}^{k-1}\{\log{\widetilde{W}_{\ell,N}(x)},\log{\widetilde{W}_{\ell,N}(y)}\geq \alpha_\ell\}\cap  \cap_{\ell = k}^{(1-\varepsilon_1)M} \{\log{\widetilde{W}_{\ell,N}(y)}\geq \alpha_\ell\}\cap \mathcal{E}_{M,N}^c \Big).
\]
By Lemma \ref{lem:combLem}, with $\widehat{q}_k$ as in \eqref{def:q_khat}, and successively conditioning on $\mathcal{F}_j$, $j=(1-\varepsilon_1)M-1,\ldots,k-1$, this is less than
\[
\mathbb{P}\left(\cap_{\ell=1}^{k-1}\{\log{\widetilde{W}_{\ell,N}(x)},\log{\widetilde{W}_{\ell,N}(y)}\geq \alpha_\ell\};\mathcal{E}_{k-1,N}^c\right)  
N^{\frac 3 M} e^{\sum_{\ell = k}^{(1-\varepsilon_1)M}(- \widehat{q}_\ell \alpha_\ell + \lambda^2_{\frac{\ell-1}{M},\frac{\ell}{M}} \widehat{q}_\ell^2/2)},
\]
where $\exp\{- \widehat{q}_\ell \alpha_\ell +  \lambda^2_{(\ell-1)/M,\ell/M} \widehat{q}_\ell^2/2 \}= e^{-\frac{(1-\delta)(1-3\delta)}{M}\log N}$.
For $\ell\leq k-1$, by independence,
\al{
\mathbb{P}\left( \log{\widetilde{W}_{\ell,N}(x)},\log{\widetilde{W}_{\ell,N}(y)}\geq \alpha_{\ell}\middle|~\kF_{\ell-1}\right)
= \mathbb{P}( \log{\widetilde{W}_{\ell,N}(x)}\geq \alpha_\ell|~\kF_{\ell-1})^{2},
}
which by Lemma \ref{lem:combLem} is less than
$N^{4/M^2}e^{-2\frac{(1-\delta)(1-3\delta)}{M}\log N}$. We get 
\[
    \mathbb{P}(\widetilde{\mathcal{A}}_{M}(x)\cap \widetilde{\mathcal{A}}_{M}(y)\cap \mathcal{E}_{M,N}^c) \leq N^{7/M}e^{-M^{-1}((1-\delta)(1-3\delta)((1-\varepsilon_1)M-k+1+2(k-1)))\log N}.
\]
We conclude by taking $M$ large and $\delta$ small.
\end{proof}
\begin{proof}[Proof of \eqref{Eq: lower bound for single}]
The proof uses Lemma \ref{lem: lower bound of moments} below, which in turn is based on 
Appendix \ref{appB} and \cite{cosco2023momentsLB}.
     We estimate for any $q_k\geq 0$ with $k\in \Iintv{1,(1-\varepsilon_1)M}$,
\aln{
&\mathbb{P}(\log \widetilde{W}_{k,N}\geq \alpha_k\quad\forall k\in \Iintv{1,(1-\varepsilon_1)M})\label{Eq: lower estimate1}\\
&\geq  {\mathbb{E}\Big[ e^{-\sum_{\ell=1}^{(1-\varepsilon_1)M} q_\ell \log \widetilde{W}_{\ell,N}} e^{\sum_{\ell=1}^{(1-\varepsilon_1)M} q_\ell \log \widetilde{W}_{\ell,N}} \mathbf{1}_{\{ \log \widetilde{W}_{k,N}\in [\alpha_k,\alpha_k' ],\forall k\in \Iintv{1,(1-\varepsilon_1)M}\}}\Big]}\notag\\
&\geq e^{-\sum_{k=1}^{(1-\varepsilon_1)M} q_k \alpha_k'} \mathbb{E}\Big[\prod_{k=1}^{(1-\varepsilon_1)M} \widetilde{W}_{k,N}^{q_k}\Big] \notag\\
& \qquad \qquad \qquad \qquad \qquad 
\times  \frac{\mathbb{E}\Big[ \prod_{k=1}^{(1-\varepsilon_1)M} \widetilde{W}_{k,N}^{q_k}\,\mathbf{1}_{\{ \log \widetilde{W}_{k,N}\in [\alpha_k,\alpha_k'],\,\forall k\in \Iintv{1,(1-\varepsilon_1)M} \}}\Big]}{\mathbb{E}\Big[\prod_{k=1}^{(1-\varepsilon_1)M} \widetilde{W}_{k,N}^{q_k}\Big]}.\notag
}
 We fix $k\in \Iintv{1,(1-\varepsilon_1)M}$ arbitrary. We have
\al{
&\mathbb{E}\Big[  e^{\sum_{\ell=1}^{(1-\varepsilon_1)M} q_\ell \log \widetilde{W}_{\ell,N}} \mathbf{1}_{ \log \widetilde{W}_{k,N}> \alpha_k' }\Big]\\
&= \mathbb{E}\Big[  e^{(1+\delta)q_k\log \widetilde{W}_{k,N}+\sum_{\ell\neq k} q_\ell \log \widetilde{W}_{\ell,N}}\, e^{-((1+\delta)q_k-q_k) \log \widetilde{W}_{k,N}} \mathbf{1}_{\log  \widetilde{W}_{k,N}> \alpha_k'}\Big]\\
&\leq  e^{-((1+\delta)q_k-q_k) \alpha_k'} \mathbb{E}\Big[  e^{(1+\delta)q_k\log \widetilde{W}_{k,N}+\sum_{\ell\neq k} q_\ell \log \widetilde{W}_{\ell,N}}\Big]\\
&\leq e^{-((1+\delta)q_k-q_k) \alpha_k'} \mathbb{E}\Big[  e^{(1+\delta)q_k\log {W}_{k,N}+\sum_{\ell\neq k} q_\ell \log {W}_{\ell,N}}\Big].
}
For any $\varepsilon'>0$, if $M$ is large enough, by Corollary~\ref{cor: VariantMomentEstimate}, then this can be further bounded from above by 
\al{
& \exp{(- \delta q_k \alpha_k')}  \\
&\times \exp{\Big(\frac{1}{2M}\Big(((1+\delta) q_k)^2  \lambda_{\tfrac{k-1}{M},\tfrac{k}{M}}^2 + 
\sum_{\ell\in\Iintv{1,(1-\varepsilon_1)M} \setminus\{k\}}q^2_\ell  \lambda_{\tfrac{k-1}{M},\tfrac{k}{M}}^2 \Big)+\varepsilon' \log N\Big)}\\
&= \exp{\Big(-\delta q_k \alpha_k' + \frac{1}{2M} (2\delta+\delta^2) q^2_k  \lambda_{\tfrac{k-1}{M},\tfrac{k}{M}}^2 + \frac{1}{2M} \sum_{\ell=1}^{(1-\varepsilon_1)M} q^2_\ell  \lambda_{\tfrac{k-1}{M},\tfrac{k}{M}}^2+\varepsilon' \log N\Big).}
}
We note that 
\al{
 & M  ( \log N)^{-1} \left( -\delta q_k \alpha_k' + (2\delta  +\delta^2) q_k^2  \lambda_{(k-1)/M,k/M}^2 /2\right)\\
 &= -2 \delta (1-\delta) + (2\delta  +\delta^2) (1-\delta)^2 =
 -  \delta^2  + \delta^4,
}
which is negative if $\delta$ is less than $1$. 

Similarly,  we estimate
\al{
&\mathbb{E}\Big[  e^{\sum_{\ell=1}^{(1-\varepsilon_1)M} q_\ell \log \widetilde{W}_{\ell,N}} \mathbf{1}_{ \log \widetilde{W}_{k,N}\leq \alpha_k}\Big]\leq e^{\delta \alpha_k q_k} \mathbb{E}\Big[  e^{\sum_{\ell\neq k} q_\ell \log \widetilde{W}_{\ell,N} + (1-\delta)q_k \log \widetilde{W}_{k,N}}\Big]\\
&\leq \exp\Big(\delta \alpha_k q_k+\sum_{\ell\neq k}q^2_\ell  \lambda_{(\ell-1)/M,\ell/M}^2 /2+((1-\delta)q_k)^2  \lambda_{(k-1)/M,k/M}^2 /2 +\varepsilon' \log N\Big)\\
& =   \exp\Big(\delta \alpha_k q_k-(q_k^2-((1-\delta)q_k)^2)  \lambda_{(k-1)/M,k/M}^2 /2 \\
&\qquad \qquad +\sum_{\ell=1}^{(1-\varepsilon_1)M} q^2_\ell  \lambda_{(\ell-1)/M,\ell/M}^2/(2M) +\varepsilon' \log N\Big).
}
We note that 
\al{
&M  (\log N)^{-1}   \left(\delta\alpha_k q_k-(q^2_k - ((1-\delta)q_k)^2) \lambda_{(k-1)/M,k/M}^2 /(2M)\right) \\
& = \delta (1-2\delta)(1-\delta) - (2\delta -\delta^2) (1-\delta)^2 = 
-\delta + 2\delta^2 - 2\delta^3 + \delta^4,
}
which is negative for $\delta<1/4$. 

Putting things together with   Lemma~\ref{lem: lower bound of moments} below, if we take $\varepsilon'$ small enough depending on $\delta$, we have
\al{
\lim_{n\to\infty} \frac{\mathbb{E}\Big[  \exp{\Big(\sum_{\ell=1}^{(1-\varepsilon_1)M} q_\ell \log \widetilde{W}_{\ell,N}\Big)} \mathbf{1}_{\{ \log \widetilde{W}_{k,N}\not\in [\alpha_k,\alpha_k'],\,\exists  k\in \Iintv{1,(1-\varepsilon_1)M} \}}\Big]}{\mathbb{E}\Big[\prod_{k=1}^{(1-\varepsilon_1)M} \widetilde{W}_{k,N}^{q_k}\Big]} =0,
}
that implies for $N$ large enough depending on $\varepsilon_1,\varepsilon',M,\delta$,  we have
$$ \frac{\mathbb{E}\Big[  \exp{\Big(\sum_{\ell=1}^{(1-\varepsilon_1)M} q_\ell \log \widetilde{W}_{\ell,N}\Big)} \mathbf{1}_{\{ \log \widetilde{W}_{k,N}\in [\alpha_k,\alpha_k'],\,\forall k\in \Iintv{1,(1-\varepsilon_1)M} \}}\Big]}{\mathbb{E}\Big[\prod_{k=1}^{(1-\varepsilon_1)M} \widetilde{W}_{k,N}^{q_k}\Big]} \geq \frac{1}2 .$$
Therefore, combined with Lemma~\ref{lem: lower bound of moments}, for any $\delta>0$ small enough depending on $\varepsilon_1,\varepsilon$, \eqref{Eq: lower estimate1} can be further bounded from below by 
\al{
 &\exp{\Big(-\sum_{k=1}^{(1-\varepsilon_1)M} q_k \alpha_k' + \sum_{k=1}^{(1-\varepsilon_1)M} q_k^2  \lambda_{(k-1)/M,k/M}^2 /2 -\varepsilon' \log N\Big)}\\
 &\geq \exp{\Big(-\frac{1}{2M}\sum_{k=1}^{(1-\varepsilon_1)M} 2 (1-\delta)^2 \log N -\varepsilon' \log N\Big)} \geq N^{-1+ \varepsilon_1-\varepsilon}.
}
\end{proof}

The following proves the lower bound.
\begin{prop}
    For any $\varepsilon_1>0$, if $M$ is large enough and $\delta$ is small enough, then we have 
    \al{
    \lim_{N\to\infty}\mathbb{P}\Big(\exists x\in \Iintv{-\sqrt{N},\sqrt{N}}^2,\,\widetilde{\mathcal{A}}_{M}(x)\text{ holds}\Big)=1.
    }
    Moreover,   we have
\al{   & \lim_{N\to\infty}\mathbb{P}\Big(\exists x\in \Iintv{-\sqrt{N},\sqrt{N}}^2,\, \\
&\qquad \qquad \log Z_{N}(x) > (1-2\delta)\sqrt 2\sum_{i=1}^{(1-\varepsilon_1)M}  \lambda_{(i-1)/M,i/M} \sqrt{\log N} -\varepsilon_1\sqrt{\log N}\Big)=1.
}    
\end{prop}
\begin{proof}
We compute
    \al{
    &\mathbb{E}\left[\left(\sum_{x\in \Iintv{-\sqrt{N},\sqrt{N}}^2} \mathbf{1}_{\widetilde{\mathcal{A}}_{M}(x)} \right)^2\right]= \sum_{x,y\in \Iintv{-\sqrt{N},\sqrt{N}}^2} \mathbb{P}(\widetilde{\mathcal{A}}_{M}(x)\cap \widetilde{\mathcal{A}}_{M}(y))\\
    &\leq    \sum_{k=0}^{(1-\varepsilon_1)M}\sum_{\substack{x,y\in\Iintv{-\sqrt{N},\sqrt{N}}^2:\\ \|x-y\|\in [r_k,r_{k+1}]}} \mathbb{P}(\widetilde{\mathcal{A}}_{M}(x)\cap \widetilde{\mathcal{A}}_{M}(y))\\
    &\qquad +  \sum_{\substack{x,y\in\Iintv{-\sqrt{N},\sqrt{N}}^2:\\ \|x-y\|>r_{(1-\varepsilon_1)M+1}}} \mathbb{P}(\widetilde{\mathcal{A}}_{M}(x)\cap \widetilde{\mathcal{A}}_{M}(y)). 
    }
    For the first term to each $k\in \Iintv{1,(1-\varepsilon_1)M}$, if $M$ is large enough and $\delta>0$ small enough, by Lemma~\ref{lem: key lemma for lower bound}, then we estimate
    \al{
    &\sum_{x\in \Iintv{-\sqrt{N},\sqrt{N}}^2}
    \;\sum_{\substack{y\in\Iintv{-\sqrt{N},\sqrt{N}}^2:\\ \|x-y\|\in [r_k,r_{k+1}]}} \mathbb{P}(\widetilde{\mathcal{A}}_{M}(x)\cap \widetilde{\mathcal{A}}_{M}(y))\\
    &\leq (2\sqrt{N}+1)^2 (2 r_{k+1}+1)^2  N^{-1+\varepsilon_1-\frac{k-1}{M} + \varepsilon}\leq N^{\varepsilon_1 + 2\varepsilon}.
    }
    On the other hand, for the second term, since $\widetilde{\mathcal{A}}_{M}(x)$ and $\widetilde{\mathcal{A}}_{M}(y)$ are independent for  $\|x-y\|>r_{(1-\varepsilon_1)M+1}$, if $M$ is large enough depending on $\varepsilon,\varepsilon_1$,  by Lemma~\ref{lem: key lemma for lower bound}, we have
    \al{
   & \sum_{x\in \Iintv{-\sqrt{N},\sqrt{N}}^2} \; \sum_{\substack{y\in\Iintv{-\sqrt{N},\sqrt{N}}^2:\\ \|x-y\|>r_{(1-\varepsilon_1)M+1}}} \mathbb{P}(\widetilde{\mathcal{A}}_{M}(x)\cap \widetilde{\mathcal{A}}_{M}(y)) \\
   &= \sum_{x\in \Iintv{-\sqrt{N},\sqrt{N}}^2} \; \sum_{\substack{y\in\Iintv{-\sqrt{N},\sqrt{N}}^2:\\ \|x-y\|>r_{(1-\varepsilon_1)M+1}}} \mathbb{P}(\widetilde{\mathcal{A}}_{M}(x))\mathbb{P}(\widetilde{\mathcal{A}}_{M}(y)) \\
   &\geq (1-N^{-\varepsilon_1/2})N^2 N^{-2+2\varepsilon_1 - 2\varepsilon}= (1-N^{-\varepsilon_1/2})N^{2\varepsilon_1 - 2\varepsilon}.
    }
    Hence, if $\varepsilon$ is small enough depending on $\varepsilon_1$, we arrive at
    \al{
    \limsup_{N\to\infty} \frac{\mathbb{E}\left[\left(\sum_{x\in \Iintv{-\sqrt{N},\sqrt{N}}^2} \mathbf{1}_{\widetilde{\mathcal{A}}_{M}(x)} \right)^2\right]}{\mathbb{E}\left[\sum_{x\in \Iintv{-\sqrt{N},\sqrt{N}}^2} \mathbf{1}_{\widetilde{\mathcal{A}}_{M}(x)} \right]^2} \leq 1.
    }
    This together with {the inequality $P(Z\geq 1) \geq \frac{E[Z]^2}{E[Z^2]}$ for any non-negative, integer-valued random variable $Z$} ends the proof of the first claim.

    Next, we consider the second claim. 
    Under the event $$\tilde{\kE}_N  :=\{\exists x\in 
    \Iintv{-\sqrt{N},\sqrt{N}}^2,\,
    \widetilde{\mathcal{A}}_{M}(x)\text{ holds}\},$$ we take  $x_*$ such that $A_N(x_*)$ holds with a deterministic rule breaking ties so that $x_*$ is measurable with respect to $$\kF_{(1-\varepsilon_1)M} := \sigma(\omega(k,x)\mid~x\in \mathbb{Z}^2,\,k\leq (1-\varepsilon_1)M).$$
    We have
    \begin{align}
    &\mathbb{P}\Big(\exists x\in \Iintv{-\sqrt{N},\sqrt{N}}^2,\nonumber\\
    &\qquad \log Z_{N}(x) > (1-2\delta)\sum_{i=1}^{(1-\varepsilon_1)M}  \lambda_{(i-1)/M,i/M} \sqrt{\log N} -\varepsilon_1\sqrt{\log N}\Big)\nonumber\\
    &\geq \mathbb{E}[\mathbf{1}_{\tilde{\kE}_N\cap \{\sum_{i=(1-\varepsilon_1)M+1}^M  \log W_{i,N}(x_*) > - \varepsilon_1\sqrt{\log N}\}}]\label{eq-151024}\\
    &=  \mathbb{P}(\tilde{\kE}_N)-\mathbb{E}\Big[\mathbf{1}_{\tilde{\kE}_N}\mathbb{P}\Big(\sum_{i=(1-\varepsilon_1)M+1}^M  \log W_{i,N}(x_*) \leq - \varepsilon_1\sqrt{\log N}\Big| ~\Big.\mathcal{F}_{(1-\varepsilon_1)M}\Big)\Big],\nonumber
    \end{align}
    where the first term on the right-hand side converges to $1$ by the first claim and the second term converges to $0$ by Corollary~\ref{eq: lower tail2} since $\sum_{i=(1-\varepsilon_1)M+1}^M  \log W_{i,N}(x_*) = \log{(Z_N(x_{\star})/Z_{(1-\varepsilon_1)M}(x_{\star}))}$. 
\end{proof}
\subsection{Lower bound on joint moments}
\label{subsec-LBJM}
\begin{lem}\label{lem: lower bound of moments} For any $q_k$ satisfying \eqref{eq:conditionqk}, for any $\varepsilon>0$ and $m\in \Iintv{1,M}$, if $N$ is large enough, then we have
\begin{equation}
\mathbb{E}\left[\prod_{k=1}^m \widetilde{W}_{k,N}^{q_k}\right] \geq e^{(1- \varepsilon)  \sum_{k=1}^m \binom{q_k}{2} \lambda^2_{\frac{k-1}{M},\frac{k}{M}}}.
\end{equation}
\end{lem}
\begin{proof}
We first prove the claim with ${W}_{k,N}$  in place of $\widetilde{W}_{k,N}$. 
For $k\in \Iintv{1,M}$, let $\mu_k^{\otimes q_k}$ be the $q_k$-product of polymer measures of time horizon $t_{k-1}$. 
Assume $\varepsilon\in (0,1)$, and let $B_k$ be the ball of $\mathbb Z^2$ of radius $r_{k-1}\log N$ and $B_k^{\otimes q_k}$ be the $q_k$-times euclidean product of $B_k$. For $k=0$, define $B_0 := \{0\}$ and $\mu_0:=\delta_0$. We define
$$I_{s,t}^{\otimes q} := \exp\{\beta_N^2\sum_{n=s}^t \sum_{(i,j)\in \mathcal C_q} \mathbf{1}_{S_n^i=S_n^j}\},\, I_t^{\otimes q}:= I_{1,t}^{\otimes q},$$ and
\begin{equation}
    \label{eq-100225a}
\Gamma_k := \inf_{\mathbf{x}\in B_k^{\otimes q_k}} \rme_{\mathbf x}^{\otimes q_k} \left[I_{t_k-t_{k-1}}^{\otimes q_k} \right].
    \end{equation}
 We first prove the following:
 \begin{equation} \label{eq:lowerBoundConditioning}
 \mathbb{E}\left[\prod_{k=1}^M ({W}_{k,N})^{q_k}\right] \geq \mathbb{E}\left[\prod_{k=1}^M \mu_{k}^{\otimes q_k}(B_k^{\otimes q_k})\right] \prod_{k=1}^M \Gamma_k. 
 \end{equation}
By Markov’s property,  we have 
\begin{equation} \label{eq:conditionalExpFormula}
    \mathbb{E}\left[({W}_{k,N})^{q_k} \middle| \mathcal F_{k-1} \right]  = \sum_{\mathbf{x}\in (\mathbb Z^2)^{q_k}}\mu_{t_{k-1}}^{\otimes q_k}(\mathbf x) \rme_{\mathbf x}^{\otimes q_k} \left[I_{t_k-t_{k-1}}^{\otimes q_k}\right]\geq \mu_{t_k}^{\otimes q_k}(B_k^{\otimes q_k}) \Gamma_k.
\end{equation}
Therefore, \eqref{eq:lowerBoundConditioning} follows by conditioning successively on $\mathcal F_{M-1}, \dots, \mathcal F_{1}$.
Hence,
with $\neg A$ denoting the complementary set of $A$, 
\al{
&\mathbb{E}\left[\prod_{k=1}^m ({W}_{k,N})^{q_k}\right] \geq \Gamma_m \mathbb{E}\left[\left(\prod_{k=1}^{m-1} ({W}_{k,N})^{q_k}\right) \mu_{t_{m-1}}^{\otimes q_m}(B_m^{\otimes q_m})\right]\\
&= \Gamma_m \left(\mathbb{E}\left[\prod_{k=1}^{m-1} ({W}_{k,N})^{q_k}\right] - \mathbb{E}\left[\left(\prod_{k=1}^{m-1} ({W}_{k,N})^{q_k}\right)\mu_{t_{m-1}}^{\otimes q_m}(\neg B_m^{\otimes q_m})\right] \right).
}
Now, observe that by the union bound and exchangeability of the particles, 
$$\mu_{t_{m-1}}^{\otimes q_m}(\neg B_m^{\otimes q_m}) \leq q_m \mu_{t_{m-1}}(\neg B_m), \quad  \mbox{\rm a.s.},$$ so that by H\"older’s inequality (set $p_0^{-1} + p_1^{-1} = 1$), we obtain that
\begin{eqnarray}\label{eq:Gammammutm}&&\Gamma_m \mathbb{E}\left[\left(\prod_{k=1}^{m-1} ({W}_{k,N})^{q_k}\right)\mu_{t_{m-1}}^{\otimes q_m}(\neg B_m^{\otimes q_m})\right] \nonumber \\
&\leq &\Gamma_m q_m \mathbb{E}\left[\prod_{k=1}^{m-1} {W}_{k,N}^{q_k p_0}\right]^{1/p_0} \mathbb{E}\left[\mu_{t_{m-1}}(\neg B_m)^{p_1}\right]^{1/p_1}.
\end{eqnarray}
Since $\mu_t(A) = Z_t^{-1}\rme[e_t \mathbf{1}_A]$, the Cauchy-Schwarz inequality entails 
\begin{align*}
\mathbb{E}\left[\mu_{t_{m-1}}(\neg B_m)^{p_1}\right] \leq \mathbb{E}\left[Z_t^{-2p_1}\right]^{1/2} \rme^{\otimes 2p_1}\left[I_{t_{m-1}}^{\otimes 2p_1} \mathbf{1}_{S^1_{t_{m-1}} \notin B_m,\dots,S^{2p_1}_{t_{m-1}}\notin B_m} \right]^{1/2},
\end{align*}
so using again H\"older's inequality with $a^{-1}+b^{-1}=1$,
\begin{align*}
\mathbb{E}\left[\mu_{t_{m-1}}(\neg B_m)^{p_1}\right]&\leq \mathbb{E}\left[Z_t^{-2p_1}\right]^{1/2}  \rme^{\otimes 2p_1}\left[I_{t_m}^{\otimes 2p_1}(\sqrt a \hat \beta)\right]^{1/(2a)} p_{t_{m-1}}(\neg B_m)^{1/(2b)} \\
&\leq C e^{- (\log N)^2/C},
\end{align*}
where the last inequality holds by Theorem~\ref{eq: lower tail}, Theorem~\ref{thm:qmomentupperBound} and usual moderate deviation estimates of the simple random walk \cite[Theorem~2.3.11]{LL10}. Coming back to \eqref{eq:Gammammutm}, first observe that $\Gamma_m q_m \leq N^{C}$ by Theorem~\ref{thm:qmomentupperBound}. Then, by Corollary~\ref{cor: VariantMomentEstimate}, with $q_k'=q_k p_0$ and $p_0$ close enough to $1$,
$\mathbb{E}\left[\prod_{k=1}^{m-1} {W}_{k,N}^{q_k'}\right] 
\leq N^{C_1},$ 
for some $C_1=C_1(\hat \beta,\delta,M)$. 
This results in the following bound:
\begin{equation} \label{eq:boundJointMoments}
    \Gamma_m \mathbb{E}\left[\left(\prod_{k=1}^{m-1} ({W}_{k,N})^{q_k}\right)\mu_{t_{m-1}}^{\otimes q_m}(\neg B_m^{\otimes q_m})\right] \leq C N^{C_2} e^{- (\log N)^2/C}.
\end{equation}
Since the quantity in the right-hand side of \eqref{eq:boundJointMoments} vanishes, by repeating the same argument, we obtain
\begin{equation} \label{eq:lowerBoundGammak}
    \mathbb{E}\left[\prod_{k=1}^m ({W}_{k,N})^{q_k}\right] \geq \prod_{k=1}^M \Gamma_k - o_{N}(1).
\end{equation}
Finally, we use \cite{cosco2023momentsLB} to lower bound $\Gamma_k$ as in \eqref{eq:lowerBoundGammar} and use the fact that $\varepsilon>0$ is arbitrary.

Next, we prove the claim with $\widetilde{W}$. Let ${\eta}\in (0,\delta)$ with $1/{\eta}\in\N$. Note that for $0\leq x\leq y$,
$$y^{q_\ell} - x^{q_\ell} = (y^\eta)^{q_\ell/\eta} - (x^\eta)^{q_\ell/\eta} \leq q_{\ell} \eta^{-1}(y^\eta-x^\eta)y^{q_\ell-\eta}.$$ By H\"older's inequality, for $\eta>0$ small enough, we have
\al{
&\mathbb{E}\left[\prod_{k=1}^m ({W}_{k,N})^{q_k}\right]-\mathbb{E}\left[\prod_{k=1}^m \widetilde{W}_{k,N}^{q_k}\right]\\
&\leq \sum_{\ell=1}^m  \mathbb{E}\left[\left(\prod_{k\in \Iintv{1,m}\setminus \{\ell\}} ({W}_{k,N})^{q_k}\right)(({W}_{\ell,N})^{q_\ell}-\widetilde{W}_{\ell,N}^{q_\ell})\right]\\
&\leq \sum_{\ell=1}^m  {\eta}^{-1} q_\ell \mathbb{E}\left[\left(\prod_{k\in \Iintv{1,m}} {W}_{k,N}^{q_k-\eta\mathbf{1}_{k=\ell}}\right)({W}_{\ell,N}^{\eta}-\widetilde{W}_{\ell,N}^{\eta})\right]\\
&\leq \sum_{\ell=1}^m  \eta^{-1} q_\ell \mathbb{E}\left[\prod_{k\in \Iintv{1,m}} {W}_{k,N}^{(1-{\eta})^{-1}(q_k-\eta\mathbf{1}_{k=\ell})}\right]^{1-\eta}\mathbb{E}\left[{W}_{\ell,N}-\widetilde{W}_{\ell,N}\right]^{\eta}\leq C e^{- \tfrac{(\log N)^2}{C}},
}
with some $C=C(\hat{\beta},\eta,\delta,M)$, where we have used $|x^{\eta}-y^{\eta}|^{1/\eta}\leq |x-y|$ for $x,y\geq 0$, Lemma~\ref{lem: convert W to tilde-W} and Corollary~\ref{cor: VariantMomentEstimate} with $q_\ell' = (1-{\eta})^{-1}(q_k-\eta\mathbf{1}_{k=\ell})$ in place of $q_\ell$ in the last line.
\end{proof}
\appendix
\section{Moments for point-to-point partition functions}
\label{app-A}
The goal of this appendix is to prove the following theorem.
\begin{thm}\label{thm: ptop moment}
Suppose that $\hat{\beta}<1$. For any $p\in\N$, there exists $C=C(\hat{\beta},p)>0$ such that for any $N\in \N$ and $\mathbf{X},\mathbf{Y}\in (\mathbb{Z}^2)^p$,
\al{
    {\rm E}_{\mathbf{X}}
 \Big(\exp\Big(\beta_N^2 \sum_{n=1}^N \sum_{(i,j)\in \mathcal{C}_p}{\bf 1}_{S_n^i=S_n^j}\Big) {\bf 1}_{\{(S_N^i)_{i\in \Iintv{1,p}}=\mathbf{Y}\}}\Big) \leq C N^{-p}. 
 }
 In particular, there exists $C=C(\hat{\beta},p)>0$ such that 
 \al{
\sup_{N\in\N} \sup_{x,y\in \mathbb{Z}^2} \mathbb{E}[(Z_N(x,y) p_N(x,y))^p]\leq C N^{-p}.
 }
\end{thm}We fix $K\in \N$. Assume $N\in\N$. Let us denote by ${\rm P}^{\rm per}_x$ and ${\rm E}_x^{\rm per}$  the probability measure and its expectation for the simple random walk on $\mathbb{T}_N : = [-2\lfloor K\sqrt{N}\rfloor+1 ,2\lfloor K\sqrt{N}\rfloor+1)^2 \cap \mathbb{Z}^2$ starting at $x\in \mathbb T_N$ with periodic condition. Denote by $(S_n)_{n\geq 0} = ({S}^i_n)_{i\in \Iintv{1,p},\,n\geq 0}$ the $p$ simple random walks  on the torus $\mathbb{T}_N$.
 We define
\[I^{\otimes p}_k = e^{\beta_N^2 \sum^k_{n=1}V(S_n)},\,V(S_n)=\sum_{(i,j)\in \mathcal{C}_p}{\bf 1}_{S_n^i=S_n^j}.\]
Given   $\mathbf{X}=(x^i)_{i\in \Iintv{1,p}}\in (\mathbb{T}_N)^p,$ we apply $$\mathbb{E}[\prod_{i=1}^p Z_i]\leq \prod_{i=1}^p \mathbb{E}[|Z_i|^p]^{1/p}$$ with $Z_i:= {\rm E}^{\rm per}_{x^i}[e^{\beta_N\sum_{n=1}^{N}(\omega(n,S_n)-\beta_N^2/2)}]$ to get
\aln{\label{eq: holder points application}
{\rm E}^{{\rm per},\otimes p}_{\mathbf{X}} [I^{\otimes p}_N ] &= \mathbb{E}\Big[ \prod_{i=1}^p {\rm E}^{\rm per}_{x^i}\Big[e^{\beta_N\sum_{n=1}^{N}(\omega(n,S_n)-\beta_N^2/2)}\Big]\Big]\notag\\
&\leq \prod_{i=1}^p \E  \Big[{\rm E}^{\rm per}_{x_i}\Big[e^{\beta_N\sum_{n=1}^{N}(\omega(n,S_n)-\beta_N^2/2)}\Big]^p\Big]^{1/p}={\rm E}^{{\rm per},\otimes p}_0 [I^{\otimes p}_N ].
} 
Our first claim is the following.
\begin{lem}\label{lem: p moment periodic}
    Let $W_N^{\rm per}(x) := {\rm E}_x^{\rm per}[e^{\sum_{i=1}^N \{\beta_N\omega(i,S_i)-\beta_N^2/2\}}]$ and $W_N^{\rm per}:=W_N^{\rm per}(0)$. For any $p\in \N$, there exists $K=K(\hat{\beta},p)\in\N$ such that
\aln{
\sup_{N\in \N} \mathbb{E}[(W_N^{\rm per})^p]<\infty.
}   
\end{lem}
\begin{proof}
We set $H_p:= \E (W_N^{\rm per})^p $. Let $\tau$ be the first time when one of the particles leaves $[-K\sqrt{N}, K\sqrt{N}]^2$, i.e.,
$$\tau := \min_{1\le i\le p}\,\tau_i,\quad\text{with}\quad \tau_i:=\inf\{n\ge0:\, S^i_n\notin[-K\sqrt{N},K\sqrt{N}]^2\}.$$  Then, by the strong Markov property and $\sup_{N\in \N}{\rm E}^{\otimes p}_0 [ I^{\otimes p}_N]<\infty$, we have 
\aln{\label{eq: Hp estimate}
H_p={\rm E}^{\rm per}_0 [W_N^p]& \leq {\rm E}^{\rm per,\otimes p}_0 [ I^{\otimes p}_N\mathbf{1}_{\tau>N}] \notag\\
&\quad + \sum_{j=0}^{\lfloor \log_2 N\rfloor } {\rm E}^{\rm per,\otimes p}_0 \left[ I^{\otimes p}_\tau \mathbf{1}_{\tau  \in [2^{-j-1} N, 2^{-j} N]} {\rm E}^{\rm per,\otimes p}_{S_\tau}\left[I^{\otimes p}_{N-\tau}\right] \right]\notag\\
& \leq {\rm E}^{\otimes p}_0 [ I^{\otimes p}_N] + \sum_{j=0}^{\lfloor \log_2 N\rfloor } {\rm E}_0^{\otimes p} [I^{\otimes p}_{\tau}  \mathbf{1}_{\tau  \in [2^{-j-1} N, 2^{-j} N]}] H_p\notag\\
&\leq  C + H_p \sum_{j=0}^{\lfloor \log_2 N\rfloor } {\rm E}_0^{\otimes p} [I^{\otimes p}_{2^{-j}N}  \mathbf{1}_{\tau  \leq 2^{-j} N}] .
}
If $K\in \N$ is large enough, by the H\"older inequality, for $\epsilon>0$ small enough such that $(1+\epsilon) \hat{\beta}^2 < 1$,
\al{
\sum_{j=0}^{\lfloor \log_2 N\rfloor } {\rm E}^{\otimes p}_0 [I^{\otimes p}_{2^{-j}N}  \mathbf{1}_{\tau  \leq 2^{-j} N}]& \leq \sum_{j=0}^{\lfloor \log_2 N\rfloor } {\rm E}^{\otimes p}_0 [(I^{\otimes p}_{2^{-j}N})^{1+\epsilon}]^{1/(1+\epsilon)}  {\rm P}^{\otimes p}_0(\tau  \leq 2^{-j} N)^{\epsilon/(1+\epsilon)}\\
&\leq p {\rm E}_0^{\otimes p}[(I^{\otimes p}_{N})^{1+\epsilon}]  \sum_{j=0}^\infty e^{ -c \epsilon K^2 2^{j}} <1/2,
}
with some $c>0$ where we have used  \cite[Theorem~2.3.11]{LL10} to estimate the exit time and taken $K$ large enough in the last line. Therefore, combined with \eqref{eq: Hp estimate}, this yields $H_p\leq 2C.$
\end{proof}
\begin{proof}[Proof of Theorem~\ref{thm: ptop moment}]
    Assume $p\in \mathbb N$ and $\mathbf{X},\mathbf{Y}\in (\mathbb{T}_N)^p$. We will compute
$$A_{\mathbf{X},\mathbf{Y};N}:= {\rm E}_{\mathbf{X}}^{\rm per, \otimes p}
 \Big(\exp\Big(\beta_N^2 \sum_{n=1}^N \sum_{(i,j)\in \mathcal{C}_p}{\bf 1}_{S_n^i=S_n^j}\Big) {\bf 1}_{\{(S_N^i)_{i\in \Iintv{1,p}}=\mathbf{Y}\}}\Big).$$
It is obvious (from the trivial projection from $\mathbb{Z}^2$ to the torus $\mathbb{T}_N$) that the restriction of the space to $\mathbb{T}_N$ makes the partition function bigger, i.e., 
\aln{\label{eq: ptop and ptoline comparision}
{\rm E}^{\otimes p}_{\mathbf{X}}
 \Big(\exp\Big(\beta_N^2 \sum_{n=1}^N \sum_{(i,j)\in \mathcal{C}_p}{\bf 1}_{S_n^i=S_n^j}\Big) {\bf 1}_{\{(S_N^i)_{i\in \Iintv{1,p}}=\mathbf{Y}\}}\Big) \leq A_{\mathbf{X},\mathbf{Y};N}.
}
We can write the above as 
$$ A_{\mathbf{X},\mathbf{Y};N}=\sum_{ (x_n^i)_{n\in \Iintv{0,N},i\in \Iintv{1,p}}\in \mathcal{P}_N(\mathbf{X},\mathbf{Y}) }\prod_{n=0}^{N-1} p(\mathbf{x}_n,\mathbf{x}_{n+1})  e^{\beta_N^2 V(\mathbf{x}_{n+1})},$$
where we denote
$\mathbf{x}_n := (x^i_n)_{i\in \Iintv{1,p}}$, $p(\mathbf{x}_n,\mathbf{x}_{n+1}):=\prod_{i\in \Iintv{1,p}}p(x^i_n, x^i_{n+1}),$
 $V(\mathbf{x}_n):= \sum_{(i,j)\in \mathcal{C}_p}{\bf 1}_{x_n^i=x_n^j}$, and 
$$\mathcal{P}_N(\mathbf{X},\mathbf{Y})\!=\!\{ (x_n^i)_{n\in \Iintv{0,N},\,i\in \Iintv{1,p}} \in \!(\mathbb{Z}^2)^{p(N+1)}\!:\! \|x^i_{n+1}-x^i_n\|_1=1, \!\bf{x}_0=\bf{X},\,\bf{x}_N=\bf{Y}\}.$$
Introduce the matrix
$ Q(\mathbf{x},\mathbf{y})=p(\mathbf{x},\mathbf{y}) e^{\beta_N^2 V(\mathbf{y})}.$ Then
$$ A_{\mathbf{X},\mathbf{Y};N}= Q^N(\mathbf{X},\mathbf{Y}).$$
Now, $Q$ is a positive matrix, irreducible, 
 and thus possess a top real eigenvalue $\lambda_N$.
Let $\phi$ denote the corresponding (positive) eigenvector. Then,
\begin{equation}
  \label{eq-3}\sum_y Q(\mathbf{x},\mathbf{y})\phi(\mathbf{y})= \lambda_N \phi(\mathbf{x}).
\end{equation}

Introduce the matrix
$$ q(\mathbf{x},\mathbf{y})= \lambda_N^{-1}Q(\mathbf{x},\mathbf{y})\frac{\phi(\mathbf{y})}{\phi(\mathbf{x})},$$
and set
$$ m(\mathbf{y})= \phi(\mathbf{y})^2 e^{\beta_N^2 V(\mathbf{y})}.$$
We have from \eqref{eq-3} that
$$ \sum_y q(\mathbf{x},\mathbf{y})=\frac{1}{\lambda_N \phi(\mathbf{x})}\sum_y Q(\mathbf{x},\mathbf{y}) \phi(\mathbf{y})=1,$$
and
$$  m(\mathbf{x}) q(\mathbf{x},\mathbf{y})=\lambda_N^{-1} \phi(\mathbf{x})\phi(\mathbf{y}) p(\mathbf{x},\mathbf{y})  e^{\beta_N^2 V(\mathbf{x})+\beta_N^2 V(\mathbf{y})}=
m(\mathbf{y}) q(\mathbf{y},\mathbf{x}),$$
since $p(\mathbf{x},\mathbf{y})=p(\mathbf{y},\mathbf{x})$. Thus, $q$ is a stochastic matrix, reversible with respect to the measure $m$.

 By definition, for any $k\in \N$, we have
\al{
\lambda_N^k \phi(X) &=  \sum_{Y\in (\mathbb{T}_N)^p} Q^k(X,Y) \phi(Y)= {\rm E}^{{\rm per},\otimes p}_X[I^{\otimes p}_k \phi(S_k)].
}
In particular, since $I^{\otimes p}_k\geq 1$, it follows that 
\al{
&\lambda_N  \min_x \phi(x)\geq \min_x \phi(x)\min_{X} {\rm E}^{{\rm per},\otimes p}_{X}[I^{\otimes p}_1]  \geq \min_x \phi(x),\\
& (\lambda_N)^N \leq  {\rm E}^{{\rm per},\otimes p}_{\bar{X}}[I^{\otimes p}_N \phi(S_N)]/\phi(\bar{X}) \leq {\rm E}^{{\rm per},\otimes p}_{\bar{X}}[I^{\otimes p}_N] \leq {\rm E}^{{\rm per},\otimes p}_0 [I^{\otimes p}_N],
}
which, together with  Lemma~\ref{lem: p moment periodic},  implies $\lambda_N\geq 1$ and $\sup_N (\lambda_N)^N <\infty$.

We claim that $\phi(\mathbf{y})/\phi(\mathbf{x})$ is uniformly bounded for any $\mathbf{x}, \mathbf{y}\in (\mathbb T_N)^p$. 
We set $|\phi|_\infty = \max_x |\phi(x)|$. Let $\bar{X}\in  (\mathbb{T}_N)^p$ be such that $\phi(\bar{X})=|\phi|_\infty$.  
Let $C,\epsilon>0$ be constants  independent of $N$  such that $\lambda_N^N \in [1,C]$ 
and ${\rm E}^{{\rm per},\otimes p}_{0}[(I^{\otimes p}_N)^{1+\epsilon}]\leq C$ for any $N\in\N$. 
Note that ${\rm E}^{{\rm per},\otimes p}_{\bar{X}}[ (I^{\otimes p}_N)^{1+\epsilon}]\leq {\rm E}^{{\rm per},\otimes p}_{0}[(I^{\otimes p}_N)^{1+\epsilon}]\leq C$ 
by \eqref{eq: holder points application}.  
Letting $A_N := \{x\in (\mathbb{T}_N)^p|~\phi(x) \geq |\phi|_\infty/2C\}$, since $\lambda_N\geq 1$, we have
\al{
 |\phi|_\infty &\leq \lambda_N^N \phi(\bar{X})\\
 &= {\rm E}^{{\rm per},\otimes p}_{\bar{X}}[I^{\otimes p}_N \phi(S_N)\mathbf{1}_{S_N\notin A_N}]+{\rm E}^{{\rm per},\otimes p}_{\bar{X}}[I^{\otimes p}_N \phi(S_N)\mathbf{1}_{S_N\in A_N}]\\
 &\leq {\rm E}^{{\rm per},\otimes p}_{\bar{X}}[I^{\otimes p}_N ]|\phi|_\infty/2C + |\phi|_\infty {\rm E}^{{\rm per},\otimes p}_{\bar{X}}[ (I^{\otimes p}_N)^{1+\epsilon}] {\rm P}^{{\rm per},\otimes p}_{\bar{X}}(S_N \in A_N)^{\epsilon}\\
 &< |\phi|_{\infty}/2 +  C |\phi|_\infty {\rm P}^{{\rm per},\otimes p}_{\bar{X}}(S_N \in A_N)^{\epsilon}.
}
Thus, we have ${\rm P}^{{\rm per},\otimes p}_{\bar{X}}(S_N \in A_N) \geq (2C)^{-1/\epsilon}>0$ where the bound is independent of $N.$ Since $\bar{X}\in \mathbb{T}_N$, $p_N(0,y)$ and $p_N(x,y)$ are comparable uniformly in $x,y\in \mathbb{T}_N$, and thus,
\[
\inf_{{\mathbf{X}}\in (\mathbb{T}_N)^p } {\rm P}^{{\rm per},\otimes p}_{{\mathbf{X}}}(S_N \in A_N)\geq c,
\]
where $c$ is a positive constant independent of $N.$ Therefore, for any $\mathbf{X}\in \mathbb{T}_N^p$, we have
\al{
C \phi({\mathbf{X}})&\geq \lambda_N^N \phi(X) = {\rm E}^{{\rm per},\otimes p}_{\mathbf{X}}[I^{\otimes p}_{N}\phi(S_{N})]\\
&\geq {\rm E}^{{\rm per},\otimes p}_{\mathbf{X}}[\phi(S_{N})\mathbf{1}_{\{S_N \in A_N\}}]\\
&\geq  {\rm P}^{{\rm per},\otimes p}_{\mathbf{X}}(S_N \in A_N) |\phi|_\infty/2C \geq c |\phi|_\infty/2C.
}
Hence. $\phi({\mathbf{X}})/\phi({\mathbf{Y}})$ is bounded and bounded away from $0$ uniformly in $\mathbf{X},\mathbf{Y}\in \mathbb{Z}^2$ and $N\in \N$.

 Now, note that 
$$ A_{\mathbf{X},\mathbf{Y};N} =\lambda_N^N \frac{\phi(\mathbf{Y})}{\phi(\mathbf{X})} q^N(\mathbf{X},\mathbf{Y}).$$
We thus have that
$$  A_{\mathbf{X},\mathbf{Y};N} \leq C q^N(\mathbf{X},\mathbf{Y}).$$
 From the
exact form in \cite[Theorem 14.3]{W00}  (in the exact way as Corollary 14.5 there is proved, except that we do not need $N$ large enough), this gives a local limit theorem for $q$, which provides 
$$ A_{\mathbf{X},\mathbf{Y};N}  \leq C' N^{-p},$$
with some $C'>0$. Combined with \eqref{eq: ptop and ptoline comparision}, this ends the proof of Theorem~\ref{thm: ptop moment}.
\end{proof}
\begin{cor}\label{cor: LLT moment}
    Suppose that $\hat{\beta}<1$. For any $p\in\N$, if $\delta>0$  satisfies $(1+\delta)\hat{\beta}^2 < 1$, then there exists $C>0$ such that for any $N\in \N$ and $x,y\in \mathbb{Z}^2$,
 \al{
 \mathbb{E}[(Z_N(x,y) p_N(x,y))^p]\leq C N^{-\frac{p}{1+\delta}}p_N(x,y)^{\frac{p\delta}{1+\delta}}.
 }
\end{cor}
\begin{proof}
By H\"older's inequality with $\delta\in (0,1/2)$ such that $(1+\delta) \hat{\beta} <1$, we have
 \begin{align*}
& {\rm E}^{\otimes p}_x\left[ I_N^{\otimes p} \mathbf{1}_{\{S^i_N = y,\,\forall i\in \Iintv{1,p}\}} \right] 
\\
&\leq {\rm E}^{\otimes p}_x\left[ (I_N^{\otimes p})^{1+\delta} \mathbf{1}_{\{S^i_N = y,\,\forall i\in \Iintv{1,p}\}}\right]^{\frac{1}{1+\delta}} {\rm E}^{\otimes p}_x\left[\mathbf{1}_{\{S^i_N = y,\,\forall i\in \Iintv{1,p}\}} \right]^{\frac{\delta}{1+\delta}}  \\
&= {\rm E}^{\otimes p}_x\left[ (I_N^{\otimes p})^{1+\delta} \mathbf{1}_{\{S^i_N = y,\,\forall i\in \Iintv{1,p}\}}\right]^{\frac{1}{1+\delta}} {\rm P}_x(X_N = y)^{\frac{p \delta}{1+\delta}} 
\leq C N^{-\frac{p}{1+\delta}} p_N(x,y)^{\frac{p\delta}{1+\delta}},
\end{align*}
with some $C>0$ where we have used Theorem~\ref{thm: ptop moment} in the last line.
\end{proof}
\section{Lower bounds on intersection moments for different starting points and time horizons}
\label{appB}
The goal of this appendix is to prove that
\begin{equation} \label{eq:lowerBoundGammar}
    \Gamma_k \geq e^{(1+o_{N}(1))\binom{q_k}{2} \lambda^2_{\frac{k-1}{M},\frac{k}{M}}(\hat \beta^2)},
\end{equation}
where $\Gamma_k$ is as in \eqref{eq-100225a}, $q_k$ satisfies \eqref{eq:conditionqk}, and $k\in \llbracket 1, M\rrbracket$.
We note that \cite{cosco2023momentsLB} only provides a lower bound on ${\rm E}_0^{\otimes q}[I_N^{\otimes q}]$ (compared to our needs,
there are differences on the time horizon and on the starting point), so one needs to adapt the proof to our situation. We now explain how.

In \cite{cosco2023momentsLB}, we considered the time range $[N^{\gamma},N]$ and decomposed it into subintervals of the form $[L_r,L_{r+1}]$ with (roughly)
\begin{equation} \label{eq:defLk}
    L_r = N^{\gamma f_r},\quad f_r = e^{r\frac{\alpha}{\log N}}, \quad 0\leq r \leq (\log \gamma^{-1}) \log N / \alpha.
\end{equation}
where $\gamma\in(0,1)$ is sent to $0$ after taking the $N\to\infty$ limit. (We used the index $k$ and not $r$ in \cite{cosco2023momentsLB}, but to avoid confusion with the use of $k$ in the current article, we replaced it by $r$.)

In our case, considering $\Gamma_k$ for $2\leq k\leq M$ corresponds (by diffusivity of the walks) to looking at the time range $[N^{(1+\varepsilon)(k-1)/M},N^{k/M}]$, which is equivalent to modifying the range of $r$ in \eqref{eq:defLk} to
\begin{equation} \label{eq:rangek}
     \log\left((1+\varepsilon) \frac{k-1}{\gamma M}\right) \alpha^{-1} \log N =: K_0^k \leq r \leq K_1^k:= \log\left(\frac{k}{\gamma M}\right) \alpha^{-1} \log N,
\end{equation}
with $\gamma,\varepsilon \ll M^{-1}$. Then, the proof of 
\cite[Proposition 2.3]{cosco2023momentsLB}  gives for all $\mathbf{x}\in B_k^{\otimes q_k}$: 
\begin{equation} \label{eq:Prop2.3}
    \rme_{\mathbf x}^{\otimes q_k} \left[I_{t_k-t_{k-1}}^{\otimes q_k}\right] \geq D_N^k(\mathbf x) \prod_{r=K_0^k}^{K_1^k} \Upsilon_r,
\end{equation}
where $D_N^k(\mathbf x)=\rme_{\mathbf x}^{\otimes q_k}[\prod_{r=K_0^k}^{K_1^k} \mathbf{1}_{A_r}]$ and $A_r,\Upsilon_r$ are defined exactly as in \cite{cosco2023momentsLB}. One easily checks that the proof of Proposition 2.10 in \cite{cosco2023momentsLB} entails that  $D_N^k(\mathbf x) \geq 1-o(1)$ uniformly in $\mathbf{x}\in B_k^{\otimes q_k}$. Finally, from \eqref{eq:Prop2.3}, repeating the proof of  \cite[Section 2.2]{cosco2023momentsLB} (see in particular the top of page 14) yields:
\begin{align*}
    \binom{q_k}{2}^{-1} \log \prod_{r=K_0^k}^{K_1^k} \Upsilon_r & \geq (1+o(1)) \int_{\log\big((1+\varepsilon) \frac{k-1}{\gamma M}\big)}^{\log\big( \frac{k}{\gamma M}\big)} \frac{\hat \beta^2 \gamma e^x}{1-\hat \beta^2 \gamma e^x}\\
    &= (1+o(1))\lambda^2_{\frac{k-1}{M}(1+\varepsilon),\frac{k}{M}}(\hat \beta^2).
\end{align*}
Putting things together, we obtain \eqref{eq:lowerBoundGammar}
for  $k\geq 2$. (The case $k=1$ follows similarly by setting $K_0^1=0$ and $K_1^1$ as in \eqref{eq:rangek} and sending $\gamma$ to 0.)
\section{An upper bound on intersection counts}
\label{app-C}
Introduce the condition for $T=T(N)\leq N$ and $q=q(N)\in\mathbb N\setminus\{1\}$: 
\begin{equation} \label{eq:q,T-condition}
\limsup_{N\to\infty} \frac{1}{\log N} \frac{\hat\beta^2
}{1-\hat \beta^2 
\log T/\log N} \binom{q}{2} <1.
\end{equation}
\begin{remark} \label{rk:Condition}
When $T=N^{k/M}$, $k\leq M$ and $q=q_k(N)$, 
the condition becomes:
\begin{equation} \label{eq:conditionqk}
\limsup_{N\to\infty} \frac{1}{\log N}\binom{q_k}{2} < \frac{1-\hat \beta^2\frac{k}{M}}{\hat \beta^2} .
\end{equation}
\end{remark}
For $s\leq t \leq N$, recall (see \eqref{eq:defpsi}) that $
\psi_{s,t,q} =  \beta_N^2  \sum_{n=s}^t \sum_{(i,j)\in \mathcal C_q} \mathbf{1}_{S_n^i=S_n^j},
$ and recall $\lambda_{u,v}$ of \eqref{eq:def_lambda_T_Nbis}.
The following crucial estimate is contained in \cite{CN24}, see Theorem 2.7 therein.
\begin{thm}\label{thm:qmomentupperBound}
Let $T = T(N)\leq N$ and $q=q(N)\CC{\to\infty}$ satisfy \eqref{eq:q,T-condition}. Uniformly in  $u\leq v\leq (\log T/\log N)$, it holds 
\begin{equation}
\sup_{X \in (\mathbb Z^2)^q} \rme^{\otimes q}_X\left[e^{\psi_{N^u,N^v,q}}\right]  \leq e^{\lambda_{u,v}^2 \binom{q}{2}+o(q^2)}.    
\end{equation}
\end{thm}
\section{Lower tail concentration}
\label{app-D}
The following lower tail concentration is proved in \cite[Proposition 3.1]{CaSuZy18}.
\begin{thm}\label{eq: lower tail}
There exists $C=C(\hat{\beta})>0$ such that for any $N\in \N$ large enough, and for any  $m\leq n\leq N$ and $t\geq 1$,
    \aln{
            &\mathbb{P}(Z_{m,n}\leq t^{-1})\leq Ce^{-(\log t)^2/C}.
}
\end{thm}
\begin{cor}\label{eq: lower tail2}
There exists $C=C(\hat{\beta})>0$ such that for any $N\in \N$ large enough, and for any  $m\leq n\leq N$ and $t\geq 1$, almost surely,
    \aln{
    &\mathbb{P}\Big( \frac{Z_n}{Z_m}\leq t^{-1} \Big|~\kF_{m} \Big. \Big)  \leq C e^{-(\log t)^2/C}.
    }
\end{cor}
\begin{proof}
Since $ {Z_n}/{Z_m}= \sum_{y\in \mathbb{Z}^2}\mu_m(0,y) \theta_{m} Z_{n-m}(y)$, there exists $C=C(\hat{\beta})>0$ such that
\al{
\mathbb{P}\Big( \frac{Z_n}{Z_m}\leq t^{-1}\Big|~\kF_{m} \Big.  \Big)& \leq \mathbb{P}\Big(\sum_{y\in \mathbb{Z}^2}\mu_m(0,y)\mathbf{1}_{\{ \theta_{m} Z_{n-m}(y) \leq 2 t^{-1}\}} \geq 1/2 \Big|~\kF_{m} \Big.\Big)\\
&\leq 2 \sum_{y\in \mathbb{Z}^2}\mu_m(0,y) \mathbb{P}( \theta_{m} Z_{n-m}(y) \leq  2t^{-1}\Big|~\kF_{m} \Big.)\\
&=2\mathbb{P}( Z_{n-m} \leq  2t^{-1})\leq  C e^{-(\log t)^2/C},
}
with some $C=C(\hat{\beta})>0$ by the Markov inequality and  Theorem~\ref{eq: lower tail}.
\end{proof}
\bibliographystyle{plain}
\bibliography{biblio}
\end{document}